\documentclass[smallextended]{svjour3}       

\usepackage{latexsym,amssymb,amsgen,amsmath,amsxtra, amsfonts, amscd}
\usepackage{bbm}
\usepackage{graphics}
\usepackage{graphicx}
\usepackage{authblk}
\usepackage[usenames]{color}
\usepackage{booktabs}
\usepackage{multirow}
\usepackage{pgfplotstable}
\usepackage{pgfplots}
\usepackage{tikz}
\usepackage{slantsc}
\usepackage[normalem]{ulem}
\usepackage[caption=false]{subfig}

\definecolor{darkgreen}{rgb}{0.0,0.5,0.0}
\newcommand{\ra}[1]{\renewcommand{\arraystretch}{#1}}

\newcommand{\redcomment}[1]{\par\textsc{\color{red}** COMMENT: #1 **}\par}

\newtheorem{algorithm}[theorem]{Algorithm}

\newcommand{\reach}{\mathrm{reach}}

\def\be{\begin{equation}}
\def\ee{\end{equation}}
\def\ba{\begin{array}}
\def\ea{\end{array}}

\def\bs{\begin{slide}}
\def\es{\end{slide}}
\def\bc{\begin{center}}
\def\ec{\end{center}}
\allowdisplaybreaks[1]
\newcommand{\omitthis}[1]{}
\newcommand{\Ariadne}{\textsc{Ariadne}}
\newcommand{\Flowstar}{\textsc{Flow*}}
\newcommand{\MATLAB}{\textsc{MATLAB}}
\newcommand{\CORA}{\textsc{CORA}}

\begin{document}

\title{Higher Order Method for Differential Inclusions}
\titlerunning{Higher Order Method for DIs}

\author{Sanja \v{Z}ivanovi\'{c} Gonzalez, Pieter Collins, Luca Geretti, Davide Bresolin, Tiziano Villa}

\authorrunning{S. \v{Z}ivanovi\'{c} Gonzalez, P. Collins, L. Geretti, D. Bresolin, T. Villa}

\date{Received: date / Accepted: date}

\maketitle

\begin{abstract}
Uncertainty is unavoidable in modeling dynamical systems and it may be represented mathematically by differential inclusions. In the past, we proposed an algorithm to compute validated solutions of differential inclusions; here we provide several theoretical improvements to the algorithm, including its extension to piecewise constant and sinusoidal approximations of uncertain inputs, updates on the affine approximation bounds and a generalized formula for the analytical error. The approach proposed is able to achieve higher order convergence with respect to the current state-of-the-art. We implemented the methodology in Ariadne, a library for the verification of continuous and hybrid systems. For evaluation purposes, we introduce ten systems from the literature, with varying degrees of nonlinearity, number of variables and uncertain inputs. The results are hereby compared with two state-of-the-art approaches to time-varying uncertainties in nonlinear systems.
\end{abstract}
\section{Introduction}\label{Intro}

In this paper we present a method for computing rigorous solutions of uncertain nonlinear dynamical systems. 
Uncertainty arises due to environmental disturbances and modeling discrepancies.
The former include input and output disturbances, and noise on sensors and actuators; the latter account for the unavoidable approximation of a model with respect to the real system due to unmodelled phenomena, order reduction and parameter variations over changes of the environment and variations
over time of the modeled system. Such uncertainty and imprecision may be modeled by differential inclusions. 

Differential inclusions are a generalization of differential equations having multivalued right-hand sides
\be\label{di2}
 \dot x(t)\in F(x(t)),\,\,x(0)=x_0,
\ee
see \cite{AubinCellina1984}, \cite{Deimling1992}, \cite{Smirnov2002}.
They arise in applications in a variety of ways within robotics, engineering, physical and biological sciences. They can be used to model differential equations with discontinuities, by taking closed convex hull of the right-hand side as proposed by Filippov~\cite{Filippov1988},
 but more importantly, use cases arise from the analysis of complex or large-scale systems.
One approach to analyze a complex system is to apply model order reduction techniques to replace a high-order system of differential equations $\dot{x}=f(x)$ by a low-order system of the form $\dot{z} \in  h(z) + [-\epsilon,\epsilon]$, where $\epsilon>0$ represents the error introduced by simplifying the model (see \cite{FortunaNunnariGallo1992}).
Another way to analyze complex systems is to analyze separately their components.
When components depend on one another, we can decouple them by replacing an input from another component with noise that varies over the range of possible values, again resulting in smaller but uncertain systems (see \cite{ChenSankaranarayanan2016}).

Another important application area for differential inclusions is control theory. Assume a control system 
\be\label{controlsystem}
 \dot x(t) = f(x(t),u(t)),\,\,\,x(0)=x_0,
\ee
where $u(t)\in U$ is not completely controllable. Then, one may need to compute reachable sets corresponding to all admissible inputs 
which, under certain assumptions, is equivalent to computing the reachable set of a differential inclusion. In fact, a well-known result states that solution sets of \eqref{controlsystem} and \eqref{di2} coincide if $f(x,U)$ is continuous, $f(x,U)$ is convex for all $x$, $F(x)= f(x,U)=\bigcup_{u\in U} f(x(t),u)$, and $U$ is compact and separable. The theorem and its proof are given in \cite{AubinCellina1984} and with slight changes in the assumptions also in \cite{Nieuwenhuis1981} or \cite{Li2007}. A recent book \cite{HanCaiHuang2016} gives more insight on the application of differential inclusions in control theory.

Similarly, we obtain a differential inclusion from the time-varying system of differential equations
\be\label{noisysystem}
 \dot x(t) = f(x(t),v(t)),\quad x(0)=x_0,\quad v(t)\in V.
\ee
Although the forms of~\eqref{controlsystem} and~\eqref{noisysystem} are identical, the interpretation is different; in~\eqref{controlsystem}, the input $u(t)$ can be chosen by the designer, whereas in~\eqref{noisysystem}, the input is determined by the environment.

To reliably analyze the behavior and properties of a system, notably safety, uncertainties in the system must be taken into account when modeling, and rigorous numerical methods are necessary in order to provide guaranteed correct solutions. 
Designing numerical algorithms for computing solutions of differential inclusions rigorously, efficiently and with high precision, remains a point of current research. 

Finding the correct balance between speed and accuracy is a challenging issue that depends on the application domain. While for online applications speed is crucial, accuracy may be a matter of life and death in cases such as a robot performing laser incision on a patient. In \cite{geraldes-ijrnc2018}, modeling in such case was presented. In fact, uncertainties already arise when describing the model at hand, hence differential inclusions seem like a natural framework for this situation. Consequently, in this paper we present a method to compute over-approximation of reachable sets of differential inclusions that prioritize accuracy (higher order method) instead of speed. In particular, we are interested in obtaining a third order analytic error for input-affine systems in a single time step and explain how an arbitrarily high order can be obtained. 

\subsection{Approach Used}

Given the differential inclusion

\be\label{di1}
 \dot x(t)\in F(x(t)),\,\,x(0)=x_0,
\ee

\noindent where $F$ is a continuous set-valued map with compact and convex values, a solution is given by an
absolutely continuous function $x:[0,T]\rightarrow\mathbb{R}^n$ such that, for almost all $t\in [0,T]$,
$x(\cdot)$ is differentiable at $t$ and $\dot x(t)\in F(x(t))$.
The solution set $S_T(x_0) \subset C([0,T],\mathbb{R}^n)$ is defined as
\[
S_T(x_0)=\{x(\cdot) \in C([0,T],\mathbb{R}^n) \,\mid\, x(\cdot)\text{ is a solution of } \eqref{di1} \}.
\]
The \emph{reachable set} at time $t$, $R(x_0,t) \subset \mathbb{R}^n$, is defined as
\begin{equation*}
R(x_0,t)=\{x(t)\in\mathbb{R}^n\,| x(\cdot) \in S_t(x_0) \} .
\end{equation*}

In order to provide an over-approximation of the reachable set of (\ref{di1}), we compute solutions of an auxiliary system
\[
\dot y(t)= f(y(t),w(t)), \quad y(0)=x_0, \  w(\cdot)\in W,
\]
\noindent by finding appropriate functions $w(t)$ and set $W$, and  adding uniform error bound on the difference between the two solutions.
%
%

In our previous papers \cite{cdc2010} and \cite{nsv2019}, the algorithm for obtaining an over-approximation in such a way was presented, the derivation in the one-dimensio-\\
nal additive case with its corresponding error formula was given, cases of affine, step and sinusoidal auxuliary functions were revealed and some computational results were showcased. Here, we provide the derivation of the local error for a general input-affine system and extract formulas for the error in several cases. Namely, we present errors of $O(h), \,O(h^2)$, and $O(h^3)$ explicitly with suitable $w(t)$, and show how arbitrary higher order error could be achieved. Formulas for the local error are obtained based on Lipschitz constants, Logarithmic norm and bounds on higher-order derivatives. Computational results are more thorough providing insights on dependency on the simplification period, number of parameters and noise levels.

An important tool in the study of affine control systems~\eqref{controlsystem} is based on the Fliess expansion~\cite{Fliess1981}, in which the evolution over a time-step $h$ is expanded as a power-series in integrals of the input. A numerical method based on this approach was given in~\cite{GruneKloeden2001}. The method cannot be directly applied to study uncertain systems~\eqref{noisysystem}, since for this problem we need to compute the evolution over all possible inputs, and this point is only briefly addressed. Our method is based on a Fliess-like expansion, and extends the results of~\cite{GruneKloeden2001} by providing error estimates which are valid for all possible inputs.

We use the logarithmic norm when possible, which gives better estimates than the Lipschitz constant. The logarithmic norm was introduced independently in \cite{Dahlquist1959}, and \cite{lozinskii1962error} in order to derive error estimates to initial value problems, see also \cite{Soderlind2006}. Using the logarithmic norm is advantageous over the use of the Lipschitz constant in the sense that the logarithmic norm can have negative values, and thus, one can distinguish between forward and reverse time integration, and between stable and unstable systems. The definition of the logarithmic norm and a theorem on the logarithmic norm estimate is given in Section \ref{sec:preliminaries}.

The numerical results given in this paper were obtained using the function calculus implemented in \Ariadne~\cite{ariadne-website}, a tool for reachability analysis and verification of cyber physical systems. In particular, we use \emph{Taylor Models} for the rigorous approximation of continuous functions. A Taylor Model expresses approximations to a function in the form of a polynomial (defined over a suitably small domain) plus an interval remainder, see~\cite{MakinoBerz2003}. 

\subsection{Related Works}

One of the first algorithms for obtaining solution sets of a differential inclusion
was given in~\cite{Frankowska1991} and~\cite{puri95}. In ~\cite{Frankowska1991} they used viability kernels and in ~\cite{puri95} they considered Lipschitz differential inclusions, giving a polyhedral method for obtaining an approximation of the solution set to an arbitrary known accuracy.
In the case where $F$ is only upper-semicontinuous with compact, convex values, it is possible to compute arbitrarily accurate over-approximations to the solution set, as shown in~\cite{CollinsGraca2009}.

To date, some different techniques and various types of numerical methods have been proposed as approximations to the reachable set of a differential inclusion. Some of the early algorithms used ellipsoidal calculus \cite{KurzhanskiValyi1997}, grid-based methods  \cite{puri95,beyn2007}, optimal control \cite{BaierGerdts2009}, discrete approximations \cite{dontchev92,dontchev2002,dontchev89}, \cite{grammel2003}, and
hybrid bounding methods \cite{RamdaniMeslemCandau2009}. However, most of these algorithms are of low order and/or time costly. 

In recent years, the focus of approximating reachable set shifted to providing rigorous solutions, i.e. over-approximations of the solution set, and several algorithms have been proposed. Interval Taylor Models were used in~\cite{LinStadtherr2007} and \cite{Chen2015}; an algorithm based on comparison theorems was given in~\cite{HarwoodBarton2016}; 
support vector machines were used in \cite{RasRieWeb2017}; 
a Lohner-type algorithm was used in \cite{KapelaZgliczyski2009} and \cite{RunggerReissig2017};
conservative linearization was used in \cite{AlthoffStursbergBuss2008}; a set-oriented method in~\cite{DellnitzKlusZiessler2017}, and polynomialization was used in \cite{RunggerZamani2018} and \cite{Althoff2013}. 
Among these, most suitable for comparison are \cite{RunggerZamani2018}, \cite{Althoff2013}, and \cite{Chen2015}, since \cite{RunggerZamani2018}  provides convergence analysis, \cite{Althoff2013} is higher-order method, and \cite{Chen2015} uses the same function calculus as our method (Taylor Models). However, only \cite{Althoff2013} and \cite{Chen2015} are implemented in state-of-the-art tools similar to $\Ariadne$, i.e. $\CORA$ and $\Flowstar$, respectively. 

Hence, in this paper we demonstrate efficiency and accuracy of our algorithm by testing ten nonlinear systems of different sizes and inputs, and we compare reachable sets that we obtain with the ones produced by $\Flowstar$ and $\CORA$. Moreover, we thoroughly test capabilities of our algorithm implemented in $\Ariadne$ by showing dependency of the results on the noise level, the number of parameters, and the simplification period. 

The paper is organized as follows. In Section \ref{sec:preliminaries}, we give the key ingredients of the theory used. In Section \ref{sec:approximation_scheme}, we give the mathematical setting for obtaining over-approximations of the reachable sets of input-affine differential inclusion; we derive the local error; we give formulas for obtaining the error of second and third orders, and show how to obtain the error of higher-orders. Implementation aspects are presented in Section \ref{sec:implementation} and thorough numerical testing of the algorithm and its comparison to other tools is presented in Section \ref{sec:results}. Finally, we conclude the paper with a summary of the results and a discussion on future research directions in Section \ref{sec:conclusions}.

\section{Preliminaries}
\label{sec:preliminaries}

Below we include relevant results needed to support the novel theory presented. As already mentioned, differential inclusions can be viewed as time-varying systems, and time-varying systems be viewed as
differential inclusions. The following theorem states conditions under which the solution sets coincide.
\begin{theorem}
 Let $f:X\times U\rightarrow X$ be continuous where $U$ is a compact separable
metric space and assume that there exists an interval $I$ and an absolutely continuous
$x:I\rightarrow \mathbb{R}^n$, such that for almost all $t\in I$,
\[
 \dot x(t)\in f(x(t),U).
\]
\noindent Then there exists a Lebesgue measurable $u:I\rightarrow U$ such that for almost all $t\in I$, $x(\cdot)$ satisfies
\[
 \dot x(t)= f(x(t),u(t)).
\]
\end{theorem}
The theorem and the proof can be found
in \cite[Corollary 1.14.1]{AubinCellina1984}.
For further work on the theory of differential inclusions see
\cite{AubinCellina1984,Deimling1992,Smirnov2002}.

\smallskip

The following theorem on existence of solutions of differential inclusions and its proof can be
found in \cite{Deimling1992}. Also, a version of the theorem and its proof can be found in \cite{AubinCellina1984}.
\begin{theorem}
Let $D\subset\mathbb{R}^n$ and $F:[0,T]\times D\rightrightarrows \mathbb{R}^n$
be an upper semicontinuous set-valued mapping, with non-empty, compact and convex values.
Assume that $\|F(t,x))\|\le c(1+\|x\|)$, for some constant $c$, is satisfied on $[0,T]$.
Then for every $x_0 \in D$, there exists an
absolutely continuous function $x:[0,T] \rightarrow \mathbb{R}^n$, such that
$x(t_0)=x_0$ and $\dot x(t)\in F(t,x(t))$ for almost all $t\in [0,T]$.
\end{theorem}

\smallskip

\smallskip

In what follows, we shall need the multidimensional mean value theorem, which can be found in standard textbooks on real analysis, e.g., see \cite{Marsden1993}. We use the following form of the theorem:

\begin{theorem}\label{MVT}
Let $V\subset \mathbb{R}^n$ be open, and suppose that $f:\mathbb{R}^n \rightarrow \mathbb{R}^m$
is differentiable on V. If $x,x+h\in V$ and $L(x;x+h)\subseteq V$, i.e., the line between $x$ and $x+h$
belongs to $V$,
\[
 f(x+h)-f(x) = \int_{0}^{1}  Df(z(s)) ds\,\cdot  h\,
\]
\noindent where $Df$ denotes Jacobian matrix of $f$,  $z(s)=x+sh$, and integration is understood component-wise.
\end{theorem}

\smallskip

In this work, we canonically use the supremum norm for the vector norm in $\mathbb{R}^n$, i.e., for $x \in \mathbb{R}^n$, $\|x\|_\infty=\max \{|x_1|,...,|x_n|\}$.
For a function $f: D\subset \mathbb{R}^n \rightarrow \mathbb{R}$ the norm used is $\|f\|_\infty = \sup_{x\in D}\|f(x)\|_\infty$.
The corresponding matrix norm instead is
\[
 \|Q\|_{\infty} = \max_{k=1,...,n} \Bigl \{ \sum_{i=1}^n |q_{ki}| \Bigr\}.
\]

Given a square matrix $Q$ and a matrix norm $\|\cdot\|$, the \emph{logarithmic norm} is defined by
\[
 \lambda(Q)= \lim_{h\rightarrow 0^+} \frac{\|I + hQ\|-1}{h} .
\]
\noindent There are explicit formulas for the logarithmic norm for several matrix norms,
see \cite{Dahlquist1959,Hairer1987}. The formula for the logarithmic norm corresponding to the matrix norm we use is
\[
 \lambda_\infty(Q) = \max_k \{ q_{kk} + \sum_{i\neq k} |q_{ki}| \}.
\]
We then take advantage of the following theorem which uses the logarithmic norm to give an estimate between a solution of a differential equation and an almost solution.
\begin{theorem}\label{lnt}
 Let $x(t)$ satisfy the differential equation $\dot x(t)=f(t,x(t))$ with $x(t_0)=x_0$, where
$f$ is Lipschitz continuous. Suppose that
there exist functions $l(t)$, $\delta(t)$ and $\rho$ such that
$\lambda(Df(t,z(t)))\le l(t)$ for all $z(t)\in \mathrm{conv}\{x(t),y(t)\}$ and
$\|\dot y(t) - f(t,y(t))\|\le \delta(t)$, $\|x(t_0)-y(t_0)\|\le \rho$.
Then for $t\ge t_0$ we have
\[
\|y(t)-x(t)\| \le e^{\int_{t_0}^t l(s)ds}\left( \rho + \int_{t_0}^{t} e^{-\int_{t_0}^s l(r)dr} \delta(s) ds \right).
\]
\end{theorem}
The theorem is presented in \cite{Hairer1987}. 

\smallskip

Numerical computations of reachable sets of time-varying systems require a rigorous way of computing with sets and functions in an Euclidean space. A suitable calculus is given by the \emph{Taylor Models} defined in~\cite{MakinoBerz2003}.
\begin{definition}\label{tm}
Let $f: D \subset \mathbb{R}^v\rightarrow \mathbb{R}$ be a function
that is $(n+1)$ times continuously partially differentiable on an open set
containing the domain $D$. Let $x_0$ be a point in $D$ and $P$ the $n$-th
order Taylor polynomial of $f$ around $x_0$. Let $I$ be an interval such
that
\[
 f(x) - P(x-x_0) \in I \text{ for all } x\in D
\]
Then the pair $(P, I)$ is an $n$-th order Taylor Model of $f$ around $x_0$ on $D$.
\end{definition}

In $\Ariadne$, the tool where the algorithm is implemented, the underlying function calculus is the Taylor Model calculus. A full description of Taylor Models as used in \Ariadne\ is given in~\cite{Collins2010}.

\section{The Analytical Error}
\label{sec:approximation_scheme}

Our objective is to provide an over-approximation of the reachable set of the differential inclusion  
 \be
\dot x(t) \in f(x(t), V),\,\,
\label{cp}
\ee
\noindent where $x:\mathbb{R}\rightarrow \mathbb{R}^n$, 
$V\subset \mathbb{R}^m$ is a compact convex set, $f$ is continuous and $f(x,V)$ is convex for all $x\in\mathbb{R}^n$.
However, in this paper, we restrict attention to input-affine systems in the form of
\begin{equation}\label{ca}
\dot x(t) = f(x(t)) + \sum_{i=1}^{m} g_i(x(t)) v_i(t); \quad x(t_0)=x_0,
\end{equation}
where $v_i(\cdot)\in [-V_i,V_i]$ is a bounded measurable function for $i=1,\dots,m$ and $V_i>0$ for all $i$.

While {\it Taylor Model} calculus already provides us with over-approximations when performing calculations such as antiderivation, 
direct application of it to the system \eqref{cp} or \eqref{ca} is not possible since $v(\cdot)$ belongs to an infinite dimensional space. Instead, we propose to define an auxiliary system, whose time-varying inputs are finitely parameterized, and to which we can apply {\it Taylor Model} calculus to obtain over-approximations, compute the difference between the two systems, and add this difference (the analytical error) to achieve an over-approximation of the reachable set. Moreover, we desire to achieve third-order error in a single step approximation.

\subsection{Single-step approximation}\label{Step}

Given an initial set of points $X_0$, define
\be\label{rch}
R(X_0,t)=\{x(t) \mid x(\cdot)\, \textrm{is a solution of~\eqref{cp} with}\; x(0)\in X_0\}
\ee

\noindent as the reachable set at time $t$. Let $[0,T]$ be an interval of existence of (\ref{cp}). Let $0=t_0,\, t_1,\, \ldots,\, t_{n-1}, t_n=T$
be a partition of $[0,T]$, and let $h_k=t_{k+1}-t_k$.
For $x\in\mathbb{R}^n$ and $v(\cdot)\in L^\infty([t_k,t_{k+1}];\mathbb{R}^m)$,
define $\phi(x_k,v(\cdot))=x(t_{k+1})$ which is the solution of~\eqref{cp}  at time $t_{k+1}$ with $x(t_k)=x_k$. At each time step we want to compute an over-approximation $R_{k+1}$ to the set
\begin{equation*}
 \reach(R_k,t_{k},t_{k+1})=\{ \phi(x_k,v(\cdot)) \mid x_k\in R_k\text{ and } v(\cdot)\in L^\infty([t_k,t_{k+1}];\mathbb{R}^m) \} ,
\end{equation*}
where $L^\infty([t_k,t_{k+1}];\mathbb{R}^m)$ is the space of essentially bounded measurable functions from interval $[t_k,t_{k+1}]$ into $\mathbb{R}^m$.
Since $L^\infty$ is infinite-dimensional, we aim to approximate the set of all solutions by
restricting the disturbances to a finite-dimensional space.
Let a set of functions $W_k\subset C([t_k,t_{k+1}];\mathbb{R}^m)$ be parameterized as $W_k=\{w(a_k,\cdot)\,|\,a_k\in A\subset\mathbb{R}^p\}$. For example, $W_k$ can be the set of all linear functions of the form $w(a_k,t) = a_{0k} + a_{1k} t$.
We then need to find an error bound $\epsilon$ such that
\be\label{eps}
 \forall\,v_k\in L^\infty([t_k,t_{k+1}];V),\ \exists\,a_k\in A \text{ s.t. } \| \phi(x_k,v_k(\cdot)) - \phi(x_k,w(a_k,\cdot)) \| \leq \epsilon_k.
\ee
Note that we do not need to find explicitly infinitely many $a_k$ values. Instead we need to choose the correct dimension ($\mathbb{R}^p$) and provide bounds to get a desired error $\epsilon_k$.
%

We define the auxiliary system at time step $k$ by
\be\label{cpa}
  \dot y(t) = f(y(t), w(a_k,t)), \,\, y_k=y(t_k), \ t\in[t_k,t_{k+1}].
\ee
\noindent We would like to choose functions $w_k=w(a_k,\cdot):[t_k,t_{k+1}]\rightarrow \mathbb{R}$,
depending on $x(t_k)$ and $v(\cdot)$,
such that the solution of (\ref{cpa}) is an approximation of high order to the solution of (\ref{cp}).
The desired local error is of $O(h^3)$ so we can expect the global error (i.e., cumulative error for the time of computation $[0,T]$) to be roughly of $O(h^2)$.

The total local error for a time-step actually consists of two parts. The first part is the analytical error
given by~\eqref{eps}. The second part is the numerical error which is discussed in Section~\ref{sec:implementation}.
We represent the time-$t_k$ over-approximation of the reachable set
\[
R_k=\{h_k(s)\,+\,   [-\varepsilon_k,\varepsilon_k]^n \; |\; s\in[-1,+1]^{p_k}\},
\]
as a Taylor Model. Here, $h_k(s)$ is the polynomial obtained using Taylor Model calculus, $[-\varepsilon_k,\varepsilon_k]^n$ is the interval remainder, and $p_k$ is the number of parameters used in the description of $R_k$.
The inclusion $R(X_0,t_k)\subseteq R_k$ is guaranteed by this approximation scheme.

\subsection{Error derivation}\label{errde}

Consider an input-affine system as in \eqref{ca}, and let  
\begin{itemize}
 \item $f:\mathbb{R}^n\rightarrow \mathbb{R}^n$ be a $C^p$ function,
 \item each $g_i:\mathbb{R}^n\rightarrow \mathbb{R}^n$ be a $C^p$ function,
 \item $v_i(\cdot)$ be a measurable function such that
  $v_i(t) \in [-V_i,+V_i]$ for some $V_i > 0$.
\end{itemize}
Here, $p\ge 1$ depends on the desired order and will be precisely defined later.
Construct a corresponding auxiliary system as explained in equation~\eqref{cpa}, i.e.,
\begin{equation}\label{caa}
\dot y(t) = f(y(t)) + \sum_{i=1}^{m} g_i(y(t))w_i(a_k,t); \quad y(t_k)=y_k, \  t\in[t_k,t_{k+1}].
\end{equation}
We assume that $w_i(t)$, $i=1,...,m$ are continuously differentiable real-valued functions. 
The error is computed by looking at the difference between the exact solution and an approximate solution obtained from the auxiliary system. It is derived using integration by parts until a desired order (e.g., $O(h^3)$) is achieved. In what follows, $Df$ denotes the Jacobian matrix, $D^2 f$ denotes the Hessian matrix, and $\lambda(\cdot)$ denotes the logarithmic norm of a matrix defined in Section \ref{sec:preliminaries}. For convenience of notation, we write $h_k=t_{k+1}-t_k$, $t_{k+1/2}=t_k + h_k/2 = (t_k+t_{k+1})/2$, and $\hat{q}(t) = \int_{t_k}^{t} q(s)\,ds$.

The single-step error in the difference between $x_{k+1}$ and $y_{k+1}$
is derived as follows. Writing~\eqref{ca} and ~\eqref{caa} as integral equations, we obtain:
\begin{subequations}\label{picard}
\begin{align}
x(t_{k+1}) &= x(t_k) + \int_{t_k}^{t_{k+1}} f(x(t)) + \sum_{i=1}^{m} g_i(x(t)) v_i(t)\,dt ; \\
y(t_{k+1}) &= y(t_k) + \int_{t_k}^{t_{k+1}} f(y(t)) + \sum_{i=1}^{m} g_i(y(t)) w_i(t)\,dt .
\end{align}
\end{subequations}
%

Without loss of generality, we assume that $x(t_k)=y(t_k)$ for all
$k\ge 0$. To be precise, initially, we assume $x(t_0)=y(t_0)$.
After obtaining an over-approximation $R_1$ to the solution set at
time $t_1$, we use $R_1$ as the set of initial points of both the original
system~\eqref{cp} and the auxiliary one~\eqref{cpa} for the next time step. 
Thus we have $x(t_1)=y(t_1)\in R_1$. We compute $R_2$, and consider it to be
the set of initial points for both equations at time $t_2$.
Proceeding like this, we have $x(t_k)=y(t_k)$, for all $k\ge0$.
Therefore, the difference between the two systems in \eqref{picard} becomes
\begin{subequations}\label{first}
\begin{align}
x(t_{k+1})-y(t_{k+1})
&=  \int_{t_k}^{t_{k+1}} f(x(t))-f(y(t))\,dt \label{firstF}\\
 &\qquad + \sum_{i=1}^{m} \int_{t_k}^{t_{k+1}}  g_i(x(t))v_i(t)- g_i(y(t))w_i(t)\, dt.\label{firstG}
\end{align}
\end{subequations}

\vspace{\baselineskip}

Integrating by parts the term~\eqref{firstF}, we obtain
\begin{align*}
 \eqref{firstF}\ & = \Bigl[ (t-t_{k+1/2}) \bigl(f(x(t)) - f(y(t))\bigr) \Bigr]_{t_k}^{t_{k+1}} \\
&\hspace{3em} - \int_{t_k}^{t_{k+1}} (t-t_{k+1/2})\frac{d}{dt}\bigl( f(x(t))- f(y(t))\bigr)dt \\[\jot]
 & = (h_k/2) \bigl(f(x(t_{k+1}))-f(y(t_{k+1})) \bigr) \\
&\hspace{3em} - \int_{t_k}^{t_{k+1}} (t-t_{k+1/2}) \bigl( Df(x(t))\dot x(t) - Df(y(t))\dot y(t)\bigr)dt.
\end{align*}

\noindent There are two ways that we deal with term~\eqref{firstG}. First we rewrite the term inside the integral as
\begin{align*}
 g_i(x(t))v_i(t)- g_i(y(t))w_i(t) =  (g_i(x(t))-g_i(y(t)))\,w_i(t) + g_i(x(t))\,(v_i(t)-w_i(t)),
\end{align*}
\noindent and then integrate by parts the second term to obtain
\begin{subequations}\label{secondG1}
\begin{align}
&\eqref{firstG} = \sum_{i=1}^{m} \int_{t_k}^{t_{k+1}} (g_i(x(t))-g_i(y(t)))\,w_i(t)\,dt \notag \\
&\quad+ \sum_{i=1}^{m} \Bigl[g_i(x(t))(\hat{v}_i(t)- \hat{w}_i(t))\Bigr]_{t_k}^{t_{k+1}}
-\sum_{i=1}^{m}\int_{t_k}^{t_{k+1}} \frac{d}{dt}\Bigl( g_i(x(t))\Bigr)\,(\hat{v}_i(t)-\hat{w}_i(t))\, dt\notag \\
&\quad=\sum_{i=1}^{m} \int_{t_k}^{t_{k+1}} (g_i(x(t))-g_i(y(t)))\,w_i(t)\,dt \label{secondG1a} \\
&\qquad\qquad+ \sum_{i=1}^{m} g_i(x(t_{k+1}))(\hat{v}_i(t_{k+1})- \hat{w}_i(t_{k+1}))\label{secondG1b}\\
&\qquad\qquad-\sum_{i=1}^{m}\int_{t_k}^{t_{k+1}} Dg_i(x(t))\,\dot x(t)\,(\hat{v}_i(t)-\hat{w}_i(t))\, dt\label{secondG1c}
\end{align}
\end{subequations}
\noindent The second derivation is obtained just by integrating by parts,

\begin{subequations}\label{secondG}
\begin{align}
\eqref{firstG}\ &= \sum_{i=1}^{m} \Bigl[ g_i(x(t))\hat{v}_i(t) - g_i(y(t))\hat{w}_i(t) \Bigr]_{t_k}^{t_{k+1}} \notag \\
&\qquad\qquad - \sum_{i=1}^{m} \int_{t_k}^{t_{k+1}}\frac{d}{dt}\Bigl(g_i(x(t))\Bigr) \hat{v}_i(t) - \frac{d}{dt}\Bigl(g_i(y(t))\Bigr) \hat{w}_i(t)\,\, dt\notag\\
  &= \sum_{i=1}^{m} g_i(x(t_{k+1}))\hat{v}_i(t_{k+1}) - g_i(y(t_{k+1}))\hat{w}_i(t_{k+1})\label{secondGa} \\
&\qquad\qquad - \sum_{i=1}^{m} \int_{t_k}^{t_{k+1}} Dg_i(x(t)) \hat{v}_i(t) \dot{x}(t)- Dg_i(y(t))\hat{w}_i(t)\dot{y}(t) \,\, dt\label{secondGb}
\end{align}
\end{subequations}
\noindent Equations~\eqref{firstF} and~\eqref{secondG1} can be used to derive second-order local error estimates. By applying the mean value theorem (Theorem~\ref{MVT}) we obtain
\begin{equation*}
f(x(t_{k+1}))-f(y(t_{k+1})) = \int_{0}^{1} Df(z(s))ds\; \bigl(x(t_{k+1})-y(t_{k+1})\bigr)
\end{equation*}
\noindent Hence,
\begin{subequations}\label{secondF}
\begin{align}
\eqref{firstF}\ &=  (h_k/2) \int_{0}^{1} Df(z(s))ds\; \bigl(x(t_{k+1})-y(t_{k+1})\bigr) \label{secondFa}\\
&\hspace{4em} - \int_{t_k}^{t_{k+1}} (t-t_{k+1/2}) \bigl( Df(x(t))\dot x(t) - Df(y(t))\dot y(t)\bigr)\,dt. \label{secondFb}
\end{align}
\end{subequations}
Separate the second part of the integrand in~\eqref{secondFb} as
\begin{subequations}
\begin{align}
  Df(x(t))\, \dot x(t) - Df(y(t))\,\dot y(t) &= Df(x(t))\,\bigl(\dot x(t) -\dot y(t)\bigr) \label{sepFa}\\[\jot]
    &\qquad\qquad +\bigl(Df(x(t)) - Df(y(t))\bigr)\,\dot y(t). \label{sepFb}
\end{align}
\end{subequations}
The first term of the right-hand-side can be expanded using
\begin{align*}
\dot x(t) -\dot y(t)
&= f(x(t)) - f(y(t)) +  \sum_{i=1}^{m} \bigl(g_i(x(t))-g_i(y(t))\bigr)w_i(t) \\
&\hspace{12em} + \sum_{i=1}^{m} g_i(x(t))\bigl((v_i(t)-w_i(t)\bigr).
\end{align*}

\noindent Hence, we obtain
\begin{subequations}\label{thirdF}
\begin{align}
\eqref{firstF} & = (h_k/2) \int_{0}^{1} Df(z(s))ds\, (x(t_{k+1})-y(t_{k+1}))\label{thirdFa}\\
&\quad - \int_{t_k}^{t_{k+1}} (t-t_{k+1/2}) \,\, Df(x(t))\,\,(f(x(t)) - f(y(t))) \, dt\label{thirdFb}\\
&\quad - \sum_{i=1}^{m} \int_{t_k}^{t_{k+1}} (t-t_{k+1/2}) \,\, Df(x(t))\,\,(g_i(x(t))-g_i(y(t)))w_i(t) \,dt\label{thirdFc}\\
&\quad - \sum_{i=1}^{m} \int_{t_k}^{t_{k+1}} (t-t_{k+1/2}) \,\, Df(x(t))\,\,g_i(x(t))\,\,(v_i(t)-w_i(t)) \, dt,\label{thirdFd} \\
&\quad - \int_{t_k}^{t_{k+1}} (t-t_{k+1/2}) \,\, (Df(x(t)) - Df(y(t)))\,\,\dot y(t) dt\label{thirdFe}
\end{align}
\end{subequations}
where~\eqref{thirdFa} is~\eqref{secondFa},~(\ref{thirdF}b-d) comes from~\eqref{sepFa}, and~\eqref{thirdFe} comes from~\eqref{sepFb}.
Note that for any $C^1$-function $h(x)$ we can write
\begin{equation*}
 |h(x(t)) - h(y(t)) | \leq
      \| Dh(z(t)) \| \cdot | x(t) - y(t) |
\end{equation*}
where $z(t)\in \overline{\mathrm{conv}}\{x(t),y(t)\}$, i.e. the closure of the convex hull of $\{x(t),y(t)\}$. 
This will allow to obtain third-order bounds for terms~(\ref{thirdF} \,b,c,e).
In order to obtain a third-order estimate for term~\eqref{thirdFd}, a further integration by parts is needed. We obtain:
\begin{subequations}\label{thirdFp}
\begin{align}\addtocounter{equation}{3}
 \eqref{thirdFd} &= - \sum_{i=1}^{m} \Bigl[ Df(x(t)) \, g_i(x(t)) \, {\mbox{\small$\displaystyle\int_{t_k}^{t}$}} (s-t_{k+1/2}) (v_i(s)-w_i(s)) ds \Bigr]_{t_k}^{t_{k+1}} \notag \\
&\qquad\begin{aligned} &\qquad + \int_{t_k}^{t_{k+1}} \bigl(D^2f(x(t))\,g_i(x(t))+ Df(x(t))Dg_i(x(t))\bigr)\,\dot x(t) \label{thirdFdp}\\
 &\hspace{10em}   \int_{t_k}^{t} (s-t_{k+1/2})(v_i(s)-w_i(s))ds\ dt .
\end{aligned}
\end{align}
\end{subequations}

\noindent Using a derivation similar to the one used for~\eqref{thirdF}, and again the mean value theorem and integration by parts, we obtain
\begin{subequations}\label{thirdG}
\begin{align}
 &\eqref{secondGa}+\eqref{secondGb} = \sum_{i=1}^{m} \int_{0}^{1} Dg_i(z(s))ds\;\bigl(x(t_{k+1})-y(t_{k+1})\bigr)\hat{w}_i(t_{k+1})\label{thirdGa}\\
&\qquad + \sum_{i=1}^{m} g_i(x_{k+1})\bigl(\hat{v}_i(t_{k+1}) - \hat{w}_i(t_{k+1})\bigr)\label{thirdGb}\\
&\qquad - \sum_{i=1}^{m} \int_{t_k}^{t_{k+1}}  \bigl(Dg_i(x(t)) - Dg_i(y(t))\bigr)\,\dot y(t)\, \hat{w}_i(t) dt\label{thirdGc}\\
&\qquad - \sum_{i=1}^{m} \int_{t_k}^{t_{k+1}}  Dg_i(x(t))\,\bigl(f(x(t)) - f(y(t))\bigr)\, \hat{w}_i(t)\, dt\label{thirdGd}\\
&\qquad - \sum_{i=1}^{m} \int_{t_k}^{t_{k+1}}  Dg_i(x(t))\,f(x(t))\,\bigl(\hat{v}_i(t)-\hat{w}_i(t)\bigr)\label{thirdGe}\\
&\qquad - \sum_{i=1}^{m}\sum_{j=1}^{m} \int_{t_k}^{t_{k+1}}  Dg_i(x(t))\,\bigl(g_j(x(t))-g_j(y(t))\bigr)\,w_j(t)\,  \hat{w}_i(t)\,dt\label{thirdGf}\\
&\qquad - \sum_{i=1}^{m}\sum_{j=1}^{m} \int_{t_k}^{t_{k+1}}  Dg_i(x(t))\,g_j(x(t))\,\bigl(v_j(t)\hat{v}_i(t)-w_j(t)\hat{w}_i(t)\bigr) \, dt. \label{thirdGg}
\end{align}
\end{subequations}
The term~\eqref{thirdGe} can be further integrated by parts to obtain
\begin{subequations}\label{thirdGp}
\begin{align}\addtocounter{equation}{4}
&\eqref{thirdGe}  = - \sum_{i=1}^{m} \Bigl[ Dg_i(x(t)) \, f(x(t)) \, {\mbox{\small$\displaystyle\int_{t_k}^{t}$}} (\hat{v}(s)-\hat{w}(s))ds \Bigr]_{t_k}^{t_{k+1}} \notag\\
 & \, + \sum_{i=1}^{m} \int_{t_k}^{t_{k+1}}  \bigl( D^2g_i(x(t))\,f(x(t)) + Dg_i(x(t))\,Df(x(t)) \bigr) \dot{x}(t) \,(\hat{\hat{v}}_i(t)-\hat{\hat{w}}_i(t))\, dt \label{thirdGep}
\end{align}
and the term~\eqref{thirdGg} to obtain
\begin{align}\addtocounter{equation}{1}
\eqref{thirdGg}&= - \sum_{i=1}^{m}\sum_{j=1}^{m} \Bigl[  Dg_i(x(t))\,g_j(x(t))\,{\mbox{\small$\displaystyle\int_{t_k}^{t}$}} \bigl(v_j(s)\hat{v}_i(s)-w_j(s)\hat{w}_i(s)\bigr) ds \Bigr] \notag\\
&\begin{aligned} &\qquad + \sum_{i=1}^{m}\sum_{j=1}^{m} \int_{t_k}^{t_{k+1}}  \bigl( D^2g_i(x(t))\,g_j(x(t))+Dg_i(x(t))\,Dg_j(x(t)) \bigr)\,\dot{x}(t) \\[-\jot]
 &\hspace{10em} \,{\mbox{\small$\displaystyle\int_{t_k}^{t}$}} \bigl(v_j(s)\hat{v}_i(s)-w_j(s)\hat{w}_i(s)\bigr)ds \ dt.
\end{aligned}
\label{thirdGgp}
\end{align}
\end{subequations}
\noindent Equations~(\ref{thirdF}-\ref{thirdGp}) can be used to derive third-order local error estimates.

\subsection{Error Formulas}\label{formulas}

  We proceed to give formulas for the local error having different assumptions on functions $f(\cdot)$, $g_i(\cdot)$
and $w_i(\cdot)$.
We present necessary and sufficient conditions for obtaining local errors of $O(h)$, $O(h^2)$, $O(h^3)$, and give a methodology
for obtaining even higher-order errors. Moreover, we give formulas for the error calculation in several cases. 

Assume that we have a bounding box $B$ on the solutions of (\ref{ca}) and (\ref{caa}) for all $t\in [0,T]$. This is easily achievable using the Euler Method on the initial set subject to the system dynamics. Then, we can obtain constants $r$, $V_i$, $K$, $K_i$, $L$, $L_i$, $H$, $\Lambda$ such that
\be\label{bounds}
\begin{gathered}
|v_i(t)|\le V_i,\,\,|w_i(t)|\le r V_i,\,\,\|f(z(t))\|\le K,\,\,\|g_i(z(t))\|\le K_{i},\,\, \lambda(Df(z(t)))\le \Lambda,\\[\jot]
\|Df(z(t))\| \le L,\,\,\|Dg_i(z(t))\|\le L_i,\,\, \|D^2f(z(t))\| \le H,\,\,\|D^2g_i(z(t))\| \le H_i, \\[\jot]
\end{gathered}
\ee
\noindent for each $i=1,...,m$, and for
all $t\in [0,T]$, and $z(\cdot)\in B$. We also set
\[  K'={\displaystyle\sum_{i=1}^{m}} V_i\,K_i, \ \  L'={\sum_{i=1}^{m}} V_i\,L_i \ \  H'={\sum_{i=1}^{m}} V_i\,H_i. \]
When possible we estimate the difference of the solutions using the Logarithmic norm rather than the Lipschitz constant. To obtain the actual error value, we replace variables and functions by their bounds from equation \eqref{bounds}. In each of the cases, $w_i(a,\cdot)$ is a real-valued finitely-parameterized function with $a\in A\subset \mathbb{R}^N$.
In general, the number of parameters $N$ depends on the number of inputs and the order of error desired. In what follows, we denote $\varphi(x)=(e^x-1)/x$.

\subsubsection{Local error of $O(h)$}

\begin{theorem}\label{case1}
For any $k\ge 0$, and all $i=1,...,m$, if
\begin{itemize}
 \item $f(\cdot)$ is a Lipschitz continuous vector function,
 \item $g_i(\cdot)$ are continuous vector functions, and
 \item $w_i(t)=0$ on $[t_k,t_{k+1}]$,
\end{itemize}
\noindent then the local error is of $O(h)$. Moreover, a formula
for the error is:
\begin{equation}\label{orderh}
 \bigl| x(t_{k+1}) - y(t_{k+1})\bigr| \le  h_k\,K'\, \varphi(\Lambda h_k).
\end{equation}
Alternatively, we can use
\begin{equation}\label{orderhalt}
 \bigl| x(t_{k+1}) - y(t_{k+1})\bigr| \le  h_k\,\biggl(2K + K'\biggr) .
\end{equation}
\end{theorem}

\begin{proof} Since $w_i(t)=0$, we have $\dot{y}(t)=f(y(t))$. Using the bounds given in~\eqref{bounds}, we can take $l(t)=\Lambda$ in Theorem~\ref{lnt} and since
\begin{align*}
 \biggl\|\dot y(t) - \Bigl(f(y(t)) + \sum_{i=1}^m g_i(y(t))v_i(t)\Bigr)\biggr\| = \biggl\|\sum_{i=1}^m g_i(y(t))v_i(t))\biggr\| \le \sum_{i=1}^{m} K_i\,V_i=K',
\end{align*}
\noindent we can take $\delta(t)=K'$. Hence the formula~\eqref{orderh} is obtained directly from Theorem~\ref{lnt}.
Note that $\varphi(\Lambda h_k) = 1 + \Lambda h_k/2 + \cdots$ is $O(1)$,
so the local error is of $O(h)$. Equation~\eqref{orderhalt} can be obtained by noting that $\sup_{t\in [t_k,t_{k+1}]} ||f(x(t))-f(y(t))||\le 2K$. $\square$
\end{proof}

\subsubsection{Local error of $O(h^2)$}

\begin{theorem}\label{case2}
For any $k\ge 0$, and all $i=1,...,m$, if
\begin{itemize}
 \item $f(\cdot)$, $g_i(\cdot)$ are $C^1$ vector functions, and
 \item $w_i(\cdot)$ are bounded measurable functions defined on $[t_k,t_{k+1}]$ which satisfy
  \be\label{se1}
   \int_{t_k}^{t_{k+1}} v_i(t) - w_i(t) \, dt = 0,
  \ee
\end{itemize}
\noindent then an error of $O(h^2)$ is obtained.
\end{theorem}

\begin{proof} To show that the error is of $O(h^2)$, we use equations~(\ref{first},\ref{secondG1}).
The equation~\eqref{firstF} is in the desired form, i.e., of $O(h^2)$, since we can write
\[
\biggl|\int_{t_k}^{t_{k+1}} f(x(t)) - f(y(t))\,dt\biggr|
\le h \, L\; {\textstyle\sup_{t\in [t_k,t_{k+1}]}}\|x(t)-y(t)\|,
\]
\noindent and $\sup_{t\in (t_k,t_{k+1})}\|x(t)-y(t)\|$ is of $O(h)$
by Theorem \ref{lnt}. Similarly, equations~\eqref{secondG1a} and~\eqref{secondG1c} are of $O(h^2)$.
Note that the equation~\eqref{secondG1b} is zero due to
(\ref{se1}). \hfill $\square$

\end{proof}

 In order to be able to compute the errors, we need the bounds on the $w_i(\cdot)$ functions. In particular,
we can restrict $w_i(\cdot)$ to belong to a certain class of functions, such as polynomial or step functions.

\medskip

\begin{theorem}\label{case2a}
For any $k\ge 0$, and all $i=1,...,m$, if
\begin{itemize}
 \item $f(\cdot)$, $g_i(\cdot)$ are $C^1$ vector functions, and
 \item $w_i(t)$ are real-valued, constant functions defined on $[t_k,t_{k+1}]$ by
 $ w_i=\frac{1}{h_k} \int_{t_k}^{t_{k+1}} v_i(t)dt , $
\end{itemize}
\noindent then a formula for calculation of the local error is given by
\begin{align}\label{constantapproximationsecondordererror}
 \|x(t_{k+1}) - y(t_{k+1})\| \leq h_k^2\,\left(\left(K+K'\right)L'/3+2\,K'\,\left(L + L'\right)\,\varphi(\Lambda h_k)\right).
\end{align}
\end{theorem}

Before we prove the theorem, note that it is straightforward to show that with the chosen $w_i(t)$, $|w_i(t)|\leq V_i$ and  $|\hat{v}_i(t)-\hat{w}_i(t)| \leq 2V_i\,h_k$ for $t\in[t_k,t_{k+1}]$. However, we can get a slightly better bound $|\hat{v}_i(t)-\hat{w}_i(t)|\le V_i\, h_k/2$ by considering the following derivation:

Without loss of generality, assume $t\in[0,h]$, and let
\begin{align*}
 a_i(t)=\frac{1}{t}\,\int_{0}^{t} v_i(s)\,ds,\ \ \
 b_i(t)=\frac{1}{h-t}\,\int_{t}^{h} v_i(s)\,ds
\end{align*}
\noindent and define \[w_i(t)=(t\,a_i(t)\, +\,(h-t)\,b_i(t))/h.\]
Then, $w_i=w_i(t)$ is constant for all $t\in [0,h]$.
Notice that $\hat{v}_i(t)=ta_i(t)$ and $\hat{w}_i(t)=(t/h)(ta_i(t)+(h-t)b_i(t))$.
Hence, we have
\begin{align*}
 \hat{v}_i(t)-\hat{w}_i(t) &= t(h-t)(a_i(t)-b_i(t))/h,\\
 |\hat{v}_i(t)-\hat{w}_i(t)| &= t(h-t)|a_i(t)-b_i(t)|/h \le V_i\,h/2.
\end{align*}
Additionally, we can prove that 
\be\label{over3}
\int_{t_k}^{t_{k+1}}|\hat{v}_i(t)-\hat{w}_i(t)|\,dt \leq V_i\,h_k^2 /3.
\ee

\begin{proof}
To derive~\eqref{constantapproximationsecondordererror}, we obtain $\|x(t_{k+1}) - y(t_{k+1})\|$ from equations~\eqref{firstF} and~\eqref{secondG1}.
Using the bounds given in~\eqref{bounds},
it is immediate that $||\dot{x}||\leq K + \sum_{i=1}^{m} V_i\,K_i$. Take $z(t)$ to satisfy the differential equation $\dot{z}(t)=f(z(t))$.
From Theorem \ref{lnt}, we have
\[
 \|x(t)-z(t)\|,\|y(t)-z(t)\| \le h_k\,\Bigl(\sum_{i=1}^{m} K_i\,V_i\Bigr)\,\varphi(\Lambda h_k)
\]
and hence
\[
 \|x(t)-y(t)\| \le 2\,h_k\,\Bigl(\sum_{i=1}^{m} K_i\,V_i\Bigr)\,\varphi(\Lambda h_k)
\]
\noindent for $t\in [t_k,t_{k+1}]$. Using the bound in~\eqref{over3} and combining the bounds for the norms of

\begin{align*}
(\ref{firstF})&\le \int_{t_k}^{t_{k+1}} L\,\biggl( 2\,h_k\,\Bigl(\sum_{i=1}^{m} K_i\,V_i\Bigr)\,\varphi(\Lambda h_k)\biggr) dt = 2h_k^2 \,L\,K'\varphi(\Lambda h_k) \\
(\ref{secondG1a})&\le \int_{t_k}^{t_{k+1}} \Bigl(\sum_{i=1}^{m} V_i L_i \Bigr)\,\biggl( 2\,h_k\,\Bigl(\sum_{i=1}^{m} K_i\,V_i\Bigr)\,\varphi(\Lambda h_k)\biggr)dt = 2h_k^2\,L'\,K'\,\varphi(\Lambda h_k)\\
(13b)&=0\\
(\ref{secondG1c})&\le \int_{t_k}^{t_{k+1}} \sum_{i=1}^{m} L_i \Bigl( K+\sum_{j=1}^{m} V_jK_j\Bigr) \bigl| \hat{v}_i(t)-\hat{w}_i(t) \bigr|\,dt \le\frac{h_k^2}{3}\,L'(K+K')\\
\end{align*}
we get the desired formula (\ref{constantapproximationsecondordererror}). \hfill $\square$
\end{proof}

\begin{remark}
Note that as $\Lambda\rightarrow 0$, then $\frac{e^{\Lambda\, h}-1}{\Lambda\,h}\rightarrow 1$.
This is also consistent with Theorem \ref{lnt}. In fact, if $\Lambda=0$, we get
\[
 \|x(t)-y(t)\| \le 2\,h_k\,\Bigl(\sum_{i=1}^{m} K_i\,V_i\Bigr)
\]
\noindent and therefore,
\begin{align}
 \|x(t_{k+1}) - y(t_{k+1})\| \leq h_k^2\,\bigl(\left(K+K'\right)L'/3+2\,K'\,\left(L + L'\right)\bigr),
\end{align}
\noindent which is still of $O(h^2)$. Further, we will not give explicit formulas for the error when $\Lambda=0$.
\end{remark}



\begin{remark}
Computation of the local error is complicated by the fact that $|v_i(t)-w_i(t)|$ is not uniformly small.
This means that the terms $g(x)(v_i-w_i)$ must be integrated over a complete time step in order to be able to use the fact that
$\int_{t_k}^{t_{k+1}} v_i(t)\,dt = \int_{t_k}^{t_{k+1}} w_i(t)\,dt$, and this must be done \emph{without} first taking norms inside the integral.
As a result, we cannot apply results on the logarithmic norm exactly directly.
Instead, we ``bootstrap'' the procedure by applying a first-order estimate for $\|x(t)-y(t)\|$ valid for any $t\in[t_k,t_{k+1}]$.
\end{remark}

\subsubsection{Local error $O(h^2)+O(h^3)$}
\label{sec:twoparametererror}

We can attempt to improve the error bounds by allowing $w_i(t)$ to have two independent parameters.
In the general case, we shall see that this gives rise to a local error estimate containing terms of $O(h^2)$ and $O(h^3)$, rather than the anticipated pure $O(h^3)$ error.

\noindent We seek two-parameter $w_i(t)$ functions which satisfy the following pair of equations
\begin{subequations} \label{se2}
\begin{align}
      &\int_{t_k}^{t_{k+1}} v_i(t) - w_i(t) \, dt = 0; \\
     \int_{t_k}^{t_{k+1}} &(t-t_{k+1/2})\,\,(v_i(t) - w_i(t)) \, dt = 0.
\end{align}
\end{subequations}

\noindent Among the various possibilities, we found that the following three representations for $w_i(t)$ have good theoretical properties:
\begin{itemize}
\item[a)] Step-function representation in the form: 
\begin{displaymath}
   w_i(t) = \left\{
     \begin{array}{rl}
       a_{i,0} & \text{if } t_k\le t < t_{k+1/2}\\
       a_{i,1} & \text{if } t_{k+1/2} \le t \le t_{k+1},
     \end{array}
   \right.
\end{displaymath}
where $t_{k+1/2}=t_k+h/2$.
\item[b)] Affine function given as: \[w_i(t)=a_{i,0}+a_{i,1}(t-t_{k+1/2})/h_k\]
\item[c)] Sinusoidal function in the form of: \[
w_i(t)= a_{i,0}\,+ a_{i,1}\,\sin\bigl(\gamma\,(t-t_{k+{1}/{2}})/h\bigr)
\]
for $\gamma=4.1632$.
\end{itemize}

To obtain appropriate sets of input functions $w_{i}(t)$, we aim to match the moments of $v_{i}(t)$:
\begin{align*}
  \mu_{i,0} &= \textstyle \frac{1}{h} \int_{t_{k}}^{t_{k+1}} v_{i}(t) \, dt; \\ \mu_{i,1} &= \textstyle \frac{4}{h^2}  \int_{t_{k}}^{t_{k+1}} (t-t_{k+1/2}) v_{i}(t) \,  dt.
\end{align*}
These satisfy $|\mu_{i,0}|\leq V_i \ \text{and}\ |\mu_{i,1}|\leq (1-\mu_{i,0}^2/V_i^2) V_i$,
so they can be parameterized as
\begin{align*} &\mu_{i,0}=c_{i,0},\\
& \mu_{i,1}= (1-c_{i,0}^2/V_i^2) c_{i,1}  
\end{align*}
\noindent for $|c_{i,0}|,|c_{i,1}|\leq V_i$. If $w_i(\cdot)$ are step-functions in the form presented in a), then $a_{i,0}=\mu_{i,0}-\mu_{i,1}$ and $a_{i,1}=\mu_{i,0}+\mu_{i,1}$.
To obtain the exact set for parameters $a_{i,0}, \,a_{i,1}$ take 
\[
v_i(t)=\left\{
\begin{array}{c}
-V_i \text{ for }\; t\in [t_k,t_k+\tau)\\
\;\;+V_i \text{ for }\; t\in [t_k+\tau,t_{k+1}].
\end{array} 
\right.
\]
Then we get 
\begin{align*}
&\mu_{i,0}=(1-2\tau/h)V_i\\
&\mu_{i,1}=(4\tau/h-4\tau^2/h^2)V_i
\end{align*}
and hence
\begin{align*}
&a_{i,0} = (1-6\tau/h+4\tau^2/h^2)V_i\\
&a_{i,1}=(1+2\tau/h-4\tau^2/h^2)V_i , 
\end{align*}
for which we find
\[ |a_{i,0}|,|a_{i,1}| \leq 5 V_i /4 \ \text{and}\ |w_i(t)| \leq 5 V_i /4. \]
We can further re-parameterize $a_{i,0}$ and $a_{i,1}$ by taking
\begin{align*}
& a_{i,0}=V_i\bigl(c_{i,0}-(1-c_{i,0}^2)c_{i,1}\bigr)\\
& a_{i,1}=V_i\bigl(c_{i,0}+(1-c_{i,0}^2)c_{i,1}\bigr),
\end{align*}
where $c_{i,0},c_{i,1}\in[-V_i,+V_i]$. This yields precisely the parameter values corresponding to an actual input $v_i(t)$.

\medskip

If $w_i(\cdot)$ are affine functions, then solving~\eqref{se2} yields $a_{i,0}=\mu_{i,0}$ and $a_{i,1}=3\mu_{i,1}$. 
To provide exact bounds for $w_i(t)$, for a given $a_{i,0}$, we can maximize $a_{i,1}$ which gives $a_{i,1}=3(1-a_{i,0}^2/V_i^2)$
yielding the constraint
\[ a_{i,0}^2 + |a_{i,1}| / 3 \leq 1 . \]
Re-parameterizing, we can set $a_{i,0}=c_{i,0}$ and $a_{i,1}=3(1-c_{i,0}^2/V_i^2)c_{i,1}$ with $c_{i,0},c_{i,1}\in[-V_i,+V_i]$,
which then gives
\begin{equation}\label{eqn:reducedparameterdomain}
  w_i(t) = c_{i,0} + { 3(1-c_{i,0}^2/V_i^2)c_{i,1}  }\,(t-t_{k+1/2})/h_k .
\end{equation}
Hence,
\be \label{eq:polynomialquadraticparameterbounds}
  |a_{i,0}|\le V_i,\,\,|a_{i,1}|\le 3V_i(1-(a_{i,0}/V_i)^2) \ \text{and} \ |w_i(t)| \leq 5V_i/3. 
\ee
 
\medskip

\noindent Alternatively, if $w_i(t)$ are sinusoidal functions in the form given in c), then $a_{i,0}=\mu_{i,0}$ and $a_{i,1}=p(\gamma)\mu_{i,1}$ where 
\[ p(2\gamma) = \tfrac{1}{2}\gamma/\bigl( \sin(\gamma)/\gamma- \cos(\gamma) \bigr), \]
and the maximum value of $|w_{i}|$ is $(p(\gamma)+1/4p(\gamma))\, V_i$. To obtain the smallest possible maximum value we minimize $p(\gamma)+1/4p(\gamma)$ which yields $\gamma\approx 4.163152$ with $p(\gamma)\approx 1.146311$, $p(\gamma)+1/4p(\gamma)\approx1.364402$.
Hence
\[ \begin{gathered} w_{i}(t) = c_{i,0} + (1-c_{i,0}^2/V_i^2) \, c_{i,1} \, \sin(4.1632(t-t_{k+1/2})); \\  |c_{i,0}|,|c_{i,1}| \leq V_i; \ \   |w_{i}| \leq 1.3645 \, V_i. \end{gathered} \]

%

In all cases a-c) we see that $|w_{i}|\le r\,V_i$, where $r$ is a constant obtained depending on the choice of the $w_{i}(\cdot)$ functions. The bound for the local error is then given by the following theorem:

\begin{theorem}\label{case2b}
For any $k\ge 0$, and all $i=1,...,m$, if
\begin{itemize}
 \item $f(\cdot)$ is a $C^2$ vector function,
 \item $g_i(\cdot)$ are non-constant $C^2$ functions, and
 \item $w_i(t)$ are real-valued functions defined on $[t_k,t_{k+1}]$ which satisfy equations~\eqref{se2} with $|w_i(t)|\le r\,V_i$ for some constant $r\in \mathbb{R}$,
\end{itemize}
\noindent then an error of $O(h^2)$ is obtained.
The formula for the error is given by
\begin{align*}
&\left(1-L(h_k/2) - h_k\, r\,L'\right)\|x(t_{k+1}) - y(t_{k+1})\| \le (h_k^2/4) (1+r^2)\,L'\,K' \\
& +(h_k^3/4)\,(1+r)\,K'\left( (2rH'+H)\, (K +rK') + L^2 + \left( 3rL + 2r^2 L' \right)L'\right)
   \varphi(\Lambda h_k)\\
&\qquad+(h_k^3/24)(1+r)\left( K + K'  \right)  \left( 3(H\,K' + L\,L') +4(H'K+LL')\right) .
\end{align*}
\end{theorem}

\begin{proof} 
With the assumptions of the theorem, we can improve the terms (\ref{thirdFd}) and (\ref{thirdGe})
such that they become (\ref{thirdFdp}) and~\eqref{thirdGep}, which are of $O(h^3)$.
In addition, we use
\begin{align*}
 \|\dot x(t)\| & \le K + \sum_{i=1}^{m} K_i\,V_i\,=\,K+K'\\
 \|\dot y(t)\| & \le K +  r \sum_{i=1}^{m} K_i\,V_i=K+rK'\\
 \|x(t)-y(t)\| & \le h_k\,(1+r) \left(\sum_{i=1}^{m} K_i\,V_i\right)\varphi(\Lambda h_k)=h_k\,(1+r)\, K'\,\varphi(\Lambda h_k).
\end{align*}
Hence, using the bounds introduced at the beginning of the section we can estimate 
\begin{align*}
|\hat{w}_i(t)|&\le |\int_{t_k}^t w_i(s)ds|\le r\,V_i\, (t-t_k)\\
(17a)&\le \frac{h_k}{2}\,L\|x(t_{k+1})-y(t_{k+1})\|\\
(17b)&\le \frac{h_k^3}{4}\,(1+r)\,K'\,L^2\; \frac{e^{\Lambda h_k}-1}{\Lambda\,h_k}\\
(17c)&\le \frac{h_k^3}{4}\,r(1+r)\,K'\,L\,L'\;\frac{e^{\Lambda h_k}-1}{\Lambda\,h_k}\\
(17d)\rightarrow (18d)&\le \frac{h_k^3}{8}\, (1+r)\,(HK'+LL')\,(K+K')\\
(17e)&\le \frac{h_k^3}{4}\,(1+r)\,K'\,H\,(K+rK')\;\frac{e^{\Lambda h_k}-1}{\Lambda\,h_k}\\
(19a)&\le h_k\,r\,L'\,\|x(t_{k+1})-y(t_{k+1})\|\\
(19b)&= 0\\
(19c)&\le \frac{h_k^3}{2}\,(1+r)\,K'\,r\,H'(K+rK')\;\frac{e^{\Lambda h_k}-1}{\Lambda\,h_k}\\
(19d)&\le \frac{h_k^3}{2}\,r(1+r)\,K'\,L\,L' \;\frac{e^{\Lambda h_k}-1}{\Lambda\,h_k}\\
(19e) \rightarrow (20e)&\le \frac{h_k^3}{6}\,(1+r)\,(H'K+LL')\,(K+K')\\
(19f)&\le \frac{h_k^3}{2}\,r^2(1+r)\,K'\,(L')^2 \;\frac{e^{\Lambda h_k}-1}{\Lambda\,h_k}\\
(19g)&\le \frac{h_k^2}{2}\, (1+r^2)\,K'\,L'
\end{align*}
Summing all the terms and rearranging gives the desired formula for the local error. \hfill $\square$
\end{proof}

\medskip

We now show that with the assumptions of the theorem we cannot in general obtain an error of $O(h^3)$.
Specifically, we assume that $w_i(t)$ are two-parameter functions satisfying
\[ \int_{t_k}^{t_{k+1}} v_i(t)-w_i(t)\,dt = \int_{t_k}^{t_{k+1}} (t-t_{k+1/2})\,(v_i(t)-w_i(t))\,dt = 0 .\]
The following counterexample gives a system for which only $O(h^2)$ local error is possible.

\begin{example}

\noindent Consider the following input-affine system which satisfies the assumption in Theorem~\ref{case2b}:
\begin{equation*}
 \dot{x}_1 = x_2 + v_1 + x_1 v_2; \quad \dot{x}_2 = x_1 + v_2;  \quad  x(t_k)=x_k.
\end{equation*}
Take inputs
\begin{equation*}
v_1(t)=\sin\left(\frac{2\pi}{h_k}(t-t_k)\right), \qquad
v_2(t)=\cos\left(\frac{2\pi}{h_k}(t-t_k)\right).
\end{equation*}

\noindent Using~(\ref{se2}), we get $w_2(t)=0$, and $w_1(t)$ is nonzero ($w_1(t)$ can be explicitly calculated for all three functions but we do not need it), hence the auxiliary system looks like

\begin{equation*}
 \dot{y}_1 = y_2 + w_1; \quad \dot{y}_2 = y_1
\end{equation*}
As shown in the previous section, the only term which might not have order
$h_k^3$ is the term in~\eqref{thirdGg} which is reduced to

\[
 \sum_{i=1}^2 \int_{t_k}^{t_{k+1}} Dg_2(x(t))g_i(x(t))\,v_i(t) \hat{v}_2(t) dt,
\]

\noindent since $Dg_1=0$. When $i=2$, the term above is of $O(h^3)$ since $\frac{1}{2} \frac{d}{dt}(\hat{v}_i^2(t))= v_i(t)\hat{v}_i(t)$ and we can integrate by parts once more. Therefore, we are left with
\begin{align*}
\int_{t_k}^{t_{k+1}} Dg_2(x(t))g_1(x(t))\,v_1(t) \hat{v}_2(t) dt = -\frac{h_k^2}{4\pi}\,\, [1\,\,\, 0]^T ,
\end{align*}
\noindent a term of $O(h^2)$.
\end{example}

\subsubsection{Local error of $O(h^3)$}
\label{sec:twoparameteradditiveinputerror}
We showed that for a general input-affine system, a local error of order
$O(h^3)$ cannot be obtained using two-parameter approximate inputs $w_i(a_{0,i},a_{1,i},t)$. However if, in addition, we assume that $g_i(\cdot)$
are constant functions or if we have a single input, then we can obtain a local error of $O(h^3)$. 
If $g_i(\cdot)$ are constant functions, then the error calculation is equivalent to the one of an even simpler case, the so called additive noise case. The equation is then given by
\be\label{an}
\dot x(t) = f(x(t)) + v(t).
\ee
\noindent Here, $v(t)=(v_1(t),...,v_n(t))$ is vector-valued.

\begin{corollary}\label{case3a}
 For any $k\ge 0$,
\begin{itemize}
 \item if the system has additive noise,
 \item $f(\cdot)$ is a $C^2$ function, and
 \item $w_i(t)$ are real-valued functions defined on $[t_k,t_{k+1}]$ which satisfy equations~\eqref{se2} with $|w_i(t)|\le r\,V_i$, for all $i=1,...,n$ and some constant $r\in \mathbb{R}$
\end{itemize}
\noindent then an error of $O(h^3)$ is obtained:
\be\label{ine}
\begin{aligned}
\bigl( 1-(h_k/2) L \bigr)&\|x(t_{k+1})-y(t_{k+1})\| \le \frac{h_k^3}{8}\, \,(1+r)\,K'\,H\,(K+K')\\
\qquad\qquad &\,\,+ \frac{h_k^3}{4}\,(1+r)\,K'\,\Bigl(L^2\, +\, H\,(K+rK')\Bigr)\varphi(\Lambda h_k).
\end{aligned}
\ee
\end{corollary}

\noindent The formula for the error in the additive noise case is simplified
because $L'=H'=0$. If we write $||v(t)||= K'$, then the result follows directly from Theorem~\ref{case2b}.

\begin{corollary}\label{case3b}
For any $k\ge 0$, if
\begin{itemize}
\item the input-affine system has a single input, i.e., $m=1$ in~\eqref{ca}
\item $f(\cdot)$ and $g(\cdot)$ are $C^2$ functions, and
 \item $w_i(t)$ are real-valued functions defined on $[t_k,t_{k+1}]$ which satisfy equations~\eqref{se2} with $|w_i(t)|\le r\,V_i$, for all $i=1,...,n$ and some constant $r\in \mathbb{R}$
\end{itemize}
\noindent then an error of $O(h^3)$ is obtained. The formula for the local error is given by
\begin{align*}
&\left(1-L(h_k/2) - h_k\, r\,L'\right)\|x(t_{k+1}) - y(t_{k+1})\| \le \\
&(h_k^3/4)\,(1+r)\,K'\,\left( (2rH'+H)\, (K +rK') + L^2 + \left( 3rL + 2r^2 L' \right)L'\right)
    \varphi(\Lambda h_k)\\
&+(h_k^3/24)\left( K + K'  \right)  \left( (1+r)(3(H\,K' + L\,L') +4(H'K+LL'))\right.\\
&\qquad\qquad\left.+8(1+r^2)\,(H'\,K'+(L')^2) \right).
\end{align*}
\end{corollary}
\begin{proof}
The result follows since the only term which is not $O(h^3)$ in~(\ref{thirdF},\ref{thirdG}) is~\eqref{thirdGg}.
In the one-input case, this simplifies to
\[ \int_{t_k}^{t_{k+1}} Dg(x(t))\,g(x(t))\,\bigl(\hat{v}(t)\,v(t)-\hat{w}(t)\,w(t)\bigr)\,dt .\]
However, we can integrate by parts to obtain
\begin{align*}
 \eqref{thirdGg} &= \Bigl[ Dg(x(t))\,g(x(t))\,\bigl(\hat{v}(t)^2-\hat{w}(t)^2\bigr) \Bigr]_{t_k}^{t_{k+1}} \\
  &\qquad\qquad - \int_{t_k}^{t_{k+1}} D\bigl(Dg(x(t))\,g(x(t))\bigr)\,\dot{x}(t)\,\bigl(\hat{v}(t)^2-\hat{w}(t)^2\bigr)\,dt .
\end{align*}
The first term vanishes since $\hat{v}(t_{k+1})=\hat{w}(t_{k+1})$, and the second is $O(h^3)$ since $\hat{v}(t)$ and $\hat{w}(t)$ are $O(h)$.
Taking all the bounds as in Theorem \ref{case2b}, the formula is easily obtained.
\hfill $\square$ \end{proof}
 Observing the error given by equations~\eqref{thirdF} and~\eqref{thirdG}, we see that if in addition to satisfying the equations given in~\eqref{se2}, the functions $w_i(\cdot)$ also satisfy
\be\label{a3}
 \int_{t_k}^{t_{k+1}} v_i(t)\hat{v}_j(t) - w_i(t)\hat{w}_j(t)\,\,dt\, =\, 0,
\ee
\noindent then we can get an error of $O(h^3)$. The question remains as to whether we can find
functions $w_i(\cdot)$ that satisfy the conditions~(\ref{se2}) and (\ref{a3}).
Since, in this case, the functions $w_i(\cdot)$ cannot be computed independently any more,
the number of parameters of each $w_i(\cdot)$ will depend on the number of inputs.
 
\begin{theorem}
 For any $k\ge 0$, if
\begin{itemize}
 \item $f(\cdot)$, $g_i(\cdot)$ are $C^2$ real vector functions, and
 \item $w_i(a_{i,0},...,a_{i,p\!-\!1},t)$ are real-valued, defined on $[t_k,t_{k+1}]$, and satisfy
\be\label{se3}
\begin{gathered}
 \int_{t_k}^{t_{k+1}} v_i(t) - w_i(t) \, dt = 0\\
 \int_{t_k}^{t_{k+1}} (t-t_{k+1/2})\,(v_i(t) - w_i(t)) \, dt = 0\\
 \int_{t_k}^{t_{k+1}} v_i(t)\hat{v}_j(t) - w_i(t)\hat{w}_j(t)\,\,dt\, =\, 0,
\end{gathered}
\ee
\end{itemize}
\noindent for all $i,j=1,...,m$, then an error of $O(h^3)$ can be obtained.
The number of parameters $p$ in at least one $w_i(\cdot)$ must satisfy $p \geq (m+3)/2$.
\end{theorem}

\begin{proof} If we can find $w_i(t)$ functions that satisfy the above conditions, then it is obvious that the only remaining $O(h^2)$ term~\eqref{thirdGg}
can be integrated by parts once more in order to give a term of $O(h^3)$.
This follows from Theorem \ref{case2}, Corollary~\ref{case3b} and the formulae~(\ref{thirdGgp}) in Section~\ref{errde}.

Now, if $m=1$, see Corollary \ref{case3b}. To see that we can find the desired functions $w_i(\cdot)$ for $m\ge 2$, 
notice that the system of equations~\eqref{se3} consists of at most $m+m+m(m-1)/2=m(m+3)/2$ independent equations.
The third equation in~\eqref{se3} has at most $m(m-1)/2$ independent equations necessary to be zero, since for $i=j$ we have
\begin{align*}
 \int_{t_k}^{t_{k+1}} v_i(t) \hat{v}_i(t) - w_i(t) \hat{w}_i(t)\, dt&= (1/2) [\hat{v}_i^2(t_{k+1})- \hat{w}_i^2(t_{k+1})],
\end{align*}
\noindent and therefore we can integrate by parts once more to get an error of $O(h^3)$. When $j>i$ integration by parts gives
\begin{align*}
 \int_{t_k}^{t_{k+1}} v_i(t) \hat{v}_j(t) - w_i(t) \hat{w}_j(t)\, dt &= \bigl[\hat{v}_i(t)\,\hat{v}_j(t) - \hat{w}_i(t)\hat{w}_j(t)\bigr]_{t_k}^{t_{k+1}}\\
&\qquad\qquad - \int_{t_k}^{t_{k+1}} \hat v_i(t) {v}_j(t) -\hat w_i(t) {w}_j(t)\, dt
\end{align*}
where the first term vanishes since $\hat{v}_i(t_{k+1}) = \hat{w}_i(t_{k+1})$.
Thus, it is sufficient to assume that $j<i$ in~\eqref{se3}. In order to solve it, we set the same number of parameters as the number of equations. Then it is not hard to see that at least one $w_i(\cdot)$ must have  $\lceil(m+3)/2\rceil$ parameters. 
\hfill $\square$
\end{proof}

\begin{table*}\centering
\ra{1.1}
\tabcolsep=0.4cm
\begin{tabular}{@{}rrr@{}}
\toprule
\# of inputs & \# of equations = & highest degree $d$  \\
 =  & total \# of parameters = & of a $w_i$ = \\
  $ m $ & $m(m+3)/2$ & $\lceil(m+1)/2\rceil$  \\
\midrule
  1 & 2 & 1 \\
  2 & 5 & 2 \\
  3 & 9 & 2 \\
  4 & 14 & 3 \\
  5 & 20 & 3 \\
  6 & 27 & 4  \\
  10 & 65 & 5 \\
\bottomrule
\end{tabular}
\caption{Total number of parameters needed depending on the number of inputs $m$ in the system. If $w_i(\cdot)$ are polynomials, the highest degree needed for at least one $w_i(\cdot)$ is given.}
\label{tbl:num-parameters}
\end{table*}

 In Table \ref{tbl:num-parameters}, we present the total number of parameters needed depending on the number of inputs in the system. In addition, if $w_i(\cdot)$ are polynomials, we highlight the minimal degree required for at least one $w_i(\cdot)$ so that a local error of $O(h^3)$ is obtained.

\subsubsection{Higher Order Local Error}

It is possible to generalize the approach used to achieve $O(h^3)$ local error.
With additional smoothness requirements on the functions $f(\cdot)$ and $g_i(\cdot)$, we can get even higher-order local errors.
In order to simplify the notation, we set $g_0=f$ and $v_0=1$. Then the input-affine system~\eqref{ca} becomes
\[
 \dot x(t) = \sum_{i=0}^{m} g_i(x(t))v_i(t).
\]

\noindent Let $g_i\in C^p$ for all $i=0,...,m$, and denote by
\[
 \dot y(t) = \sum_{i=0}^{m} g_i(y(t))w_i(a_i,t)
\]
\noindent the corresponding auxiliary system. The local error of $O(h^{p+1})$ can be obtained if $w_i(a_i,t)$ are finitely parameterized, $a_i=(a_{i,0},...,a_{i,p})$ with $p$ being sufficiently large, and they also satisfy

\begin{subequations}
\begin{align}
 \int_{t_k}^{t_{k+1}}v_i(t)\,dt &=\int_{t_k}^{t_{k+1}} w_i(t)\,dt \label{e1} \\
 \int_{t_k}^{t_{k+1}} v_j(t) \int_{t_k}^{t} v_i(s)\,ds \ dt &= \int_{t_k}^{t_{k+1}} w_j(t) \int_{t_k}^{t} w_i(s)ds\ dt \label{e2} \\
  \int_{t_k}^{t_{k+1}} v_k(t) \int_{t_k}^{t} v_j(s) \int_{t_k}^{s}v_i(r)dr\,ds\, dt &=  \int_{t_k}^{t_{k+1}} w_k(t) \int_{t_k}^{t} w_j(s) \int_{t_k}^{s}w_i(r)dr\,ds\, dt \label{e3}
\end{align}
\begin{multline}
\int_{t_k}^{t_{k+1}} v_{i_r}(s_r) \int_{t_k}^{s_r} v_{i_{r-1}}(s_{r-1})\cdots \int_{t_k}^{s_2} v_{i_1}(s_1)\,ds_1\,\cdots\,ds_{r-1}\,ds_r = \\ \qquad\qquad \int_{t_k}^{t_{k+1}} w_{i_r}(s_r) \int_{t_k}^{s_r} w_{i_{r-1}}(s_{r-1})\cdots \int_{t_k}^{s_2} w_{i_1}(s_1)\,ds_1\,\cdots\,ds_{r-1}\,ds_r
\end{multline}
\end{subequations}
\noindent In~\eqref{e1}, it is sufficient to take $i\geq1$.
In~\eqref{e2} we can restrict to $i\geq j+1$ as explained in the previous subsection. %
Next, we can simplify equation~\eqref{e3}, to get
\begin{equation*}
 \int_{t_k}^{t_{k+1}} v_k(t) \hat{v}_j(t) \hat{v}_i(t)\,dt =  \int_{t_k}^{t_{k+1}} w_k(t) \hat{w}_j(t) \hat{w}_i(t)\,dt
\end{equation*}
and consider the equations for all $i,j,k\ge 0$, such that $j\le i$. Note that for the first two equalities above we need $m + C(m+1,2)$ equations. Here, $C(n,m)=n!/(m!\,(n-m)!)$ denotes the formula for combinations. 
For the third one, we need additional $m + 3\,C(m+2,3)$ equations, which in total gives $(m/2)(m^2+4m+7)$. In general, it is not easy to see the formula for the number of equations, but if $O(h^4)$ is desired, the number of parameters needed for at least one $w_i(\cdot)$ is $(m/2)(m^2+4m+7)$.

\section{Implementation}
\label{sec:implementation}

The algorithm used for computing the reachable set of~\eqref{di1} is:

\begin{algorithm}\mbox{}
Let $R_k=\{ h_k(s)+ [-\varepsilon_k, \varepsilon_k]^n \mid s\in[-1,+1]^{p_k} \}$ be an over-approximation of the set $R(X_0,t_k)$.
To compute an over-approximation $R_{k+1}$ of $R(X_0,t_{k+1})$:
\begin{enumerate}
\item\label{auxsys} Create the auxiliary system
\[
 \dot y(t) = f(y(t),w(a_{k},t)),\,\,x(t_k)=x_k=y_k, \;\;\;t\in [t_k,t_{k+1}], y_k \in R_k, a_k\in A.
 \]
 \item Compute the necessary bounds as presented at the beginning of Section \ref{errde}
 \item\label{errstp} Compute the uniform error bound $\epsilon_k$ which represents the distance between the two solutions, i.e., $ \| \phi(x_k,v_k(\cdot)) - \phi(x_k,w(a_k,\cdot)) \|\le \epsilon_k$
 \item\label{flwstp} Compute the flow of the auxiliary system via {\it Taylor Model} integration, i.e., obtain  $(h(s_k)+[-\varepsilon_k, \varepsilon_k]^n,a_k)$ that represents an over-approximation of the solution set (see Section \ref{sec:preliminaries} on computation in \Ariadne).
 \item\label{aplystp} 
 Compute the set $R_{k+1}$ which over-approximates $R(x_0,t_{k+1})$
as $R_{k+1}=\{ (h(s_k)+[-\varepsilon_k, \varepsilon_k]^n,a_k)+[-\epsilon_k, \epsilon_k]^n \}$, i.e., the Taylor Model obtained in step \ref{aplystp} $\pm$ the analytical error obtained in step \ref{errstp}.
 \item\label{redstp} Simplify parameters (if necessary).
\end{enumerate}
\end{algorithm}

\noindent Step~\ref{flwstp} of the algorithm produces an approximated flow $\phi(x_k,w(a_k,\cdot))$
which is guaranteed to be valid for all $x_k\in R_k$.
In practice, we cannot represent ${\phi}$ exactly, and
instead use a Taylor model approximation with a guaranteed error bound. In Step~\ref{auxsys} we have $y_k = x_k$ since the over-approximated solution at the previous step is taken as the exact set to start from.
 In Step~\ref{errstp}, we compute the uniform error bound $\epsilon_k$ and in Step~\ref{aplystp} we add it to the computed flow to obtain an over-approximation,
$R_{k+1}=\{ (h(s_k)+[-\varepsilon_k, \varepsilon_k]^n,a_k)+[-\epsilon_k, \epsilon_k]^n \}$.
Step~\ref{redstp} is crucial for the efficiency and accuracy of the algorithm, as explained below.

Note that our method only guarantees a local error of high order at the sequence of rational points $\{t_k\}$ which is {\it a priori} chosen. If one is trying to estimate the error at times $t_k<t<t_{k+1}$ for any $k$ along a \emph{particular} solution, a different formula should be used such as a logarithmic norm estimate based on Theorem~\ref{lnt}.

According to the theoretical framework, the approximation error is reduced by decreasing the step size $h$. However, when an actual implementation is concerned, other numerical aspects contribute to the quality of representation of the sets and the resulting over-approximations. In particular, the computational error, i.e., the error due to implementation of the algorithm in \Ariadne, contributes towards over-approximation of the solution set in two ways. One is due to the {\it Taylor Model} calculus used and the other due to simplification of the parameters.

In order to prevent the potential blow-up of the number of polynomial terms used in the Taylor Model, small and/or high-order terms must be ``swept'' into the uniform error bound $e$. For this purpose, \Ariadne\ introduces a {\em sweep threshold} $\sigma_\mathit{thr}$ constant that represents the minimum coefficient that a term needs in order to avoid being swept into $e$. As already discussed, an additional contribution to $e$ is the error originating from the inputs approximation, which is added to the model for each variable. Therefore, over time, $e$ becomes relatively large, ultimately causing the bounds of the represented set to diverge; to address this issue, we need to extract periodically a new parameter for each variable, thus originating $n$ new independent parameters. In particular, our experience with the implementation showed that significantly more accurate results are obtained by parameter extraction at each evolution step, introducing $n$ new parameters at each step. At the same time, each step of the proposed algorithm introduces $\ell m$ additional parameters into the description of the flow, where $\ell$ is the number of parameters required for each $w_{i,k}$: $\ell=0$ for the zero case, $\ell=1$ for the constant case, and $\ell=2$ for the affine, sinusoidal and piecewise constant cases. Summarizing, after $k$ steps we end up introducing $k (n + \ell m)$ new parameters.

Therefore it is apparent that a critical requirement for the feasibility of the algorithm is to simplify periodically the representation of the reached sets. For the purposes of this paper, we rely on the following basic simplification policy: after a number of steps $N_s$ we keep a number of parameters equal to a multiple $\beta_s$ of the parameters introduced between two simplifications. To decide which parameters to keep after the simplification, we sum the coefficients of the terms where a parameter is present: the parameters with the lowest sum are considered to have the least impact on the set representation and their terms are simplified into $e$. Increasing $\beta_s$ increases the {\em average} number of parameters during evolution, while increasing $N_s$ also affects the {\em variance} of such number since the parameters are allowed to increase in a larger number of steps.

\section{Numerical Results}
\label{sec:results}

In this Section we present the results of the implementation of our approach within \Ariadne, followed by a comparison with \Flowstar\ and \CORA\ 2018. Before that, the first Subsection explains the evaluation criteria, followed by the values chosen for the numerical parameters of the three tools and by the description of the systems to be used for evaluation.

\subsection{Evaluation criteria}

In order to evaluate the quality of the reachable set of a system, we introduce the {\em volume score} (from here on simply {\em score}) $\Sigma_V$ as
\begin{equation}
\Sigma_V = \frac{1}{\sqrt[\leftroot{-2}\uproot{2}n]{\Pi_{i=1}^n \left| B_i\right|}}
\end{equation}
where $B$ is the bounding box of a set. Given a set, the formula over-ap\-prox\-i\-mates it into a box for simplicity, evaluates its volume and normalizes on the number of variables. In particular, halving the set on each dimension yields twice the score. Without extra notation, we evaluate $\Sigma_V$ on the final set of evolution to measure the quality of the whole trace.
It must be noted that since a bounding box returns an over-approximation, this measure is not entirely reliable when used for comparisons: given two different sets with equal exact bounds, a slightly larger box may be obtained for the set having the more complex representation.
Still, it is an intuitive and affordable measure that can be used across tools with different internal representations. 

In addition to the volume score, we evaluate the performance in terms of execution time $t_x$ in seconds. In particular, the execution times are obtained using a macOS 10.14.6 laptop with an Intel Core i7-6920HQ processor, using AppleClang 10.0.1 as a compiler in the case of \Ariadne\ and \Flowstar\ executables, or running on \MATLAB\ 2018b in the case of \CORA .

Finally, all the score and execution time values in the following are rounded to the nearest least significant digit.

\subsection{Tool parameters}

In the following we provide the numerical parameters used for evaluation in the benchmark. For simplicity we used fixed reasonable values for \Ariadne. For \Flowstar\ and \CORA\ we collaborated with the developers in order to identify good values. In the case of \Flowstar, such values are fixed for all systems, while for \CORA\ they are specified based on the system; in this subsection we provide the default values, while the overridden ones are given in the next subsection.

\subsubsection{\Ariadne}
\begin{itemize}
\item Sweep threshold $\sigma_\mathit{thr}$: $10^{-8}$
\item Number of steps between simplifications $N_s$: $12$
\item Number of parameters to be kept after a simplification $\beta_s$: $6$.
\end{itemize}
Please note that while a fixed maximum polynomial order can be enforced in \Ariadne\, we focused on using only a fixed sweep threshold. This choice stems from the large number of parameters involved, whose cross-products yield terms with a large order. Preliminary experimental evaluation showed that discarding polynomial terms with a small coefficient returns a better quality vs efficiency figure than discarding polynomial terms with high order (or using a combination of both strategies).
Since $N_s$ and $\beta_s$ have been introduced in this paper in order to handle the representation of sets in the presence of differential inclusions, in this section we will also show how varying their values affects the quality of such representation.

\subsubsection{\Flowstar}
\begin{itemize}
\item Mantissa precision: $53$ bits
\item Taylor model fixed order: $6$
\item Cutoff threshold: $10^{-10}$
\item Remainder estimation: $0.1$.
\end{itemize}
\subsubsection{\CORA}
\begin{itemize}
\item \texttt{zonotopeOrder}: 100 
\item \texttt{tensorOrder}: 3
\item \texttt{errorOrder}: 25
\item \texttt{intermediateOrder}: 100
\item \texttt{taylorTerms}: 5 
\item \texttt{advancedLinErrorComp}: 0
\item \texttt{reductionInterval}: inf
\item \texttt{reductionTechnique}: 'girard'
\item \texttt{maxError}: as large as possible to avoid splitting
\end{itemize}

\subsection{Benchmark Suite}

We now present ten different systems taken from the literature, with varying nonlinearity. In particular, two of them have been used when presenting time-varying uncertainties in \Flowstar. 

\begin{table*}\centering
\ra{1.1}
\tabcolsep=0.25cm
\begin{tabular}{@{}rrrrrrrrrr@{}}
\toprule
Name & Alias & Ref & n & m & $\bar{O}$ & + & $h$ & $T_e$ & steps \\ 
\midrule
Higgins-Sel'kov & HS & \cite{ChenSankaranarayanan2016} & 2 & 3 & 3 & N & 1/50 & 10 & 500 \\
Chemical Reactor & CR & \cite{HarwoodBarton2016} & 4 & 3 & 2 & N & 1/16 & 10 & 160 \\
Lotka-Volterra & LV & \cite{HarwoodBarton2016} & 2 & 2 & 2 & N & 1/50 & 10 & 500 \\
Jet Engine & JE & \cite{Chen2015} & 2 & 2 & 2 & Y & 1/50 & 5 & 250 \\
PI Controller & PI & \cite{Chen2015} & 2 & 1 & 2 & Y & 1/32 & 5 & 160 \\
Jerk Eq. 21 & J21 & \cite{Sprott1997} & 3 & 1 & 5/3 & N & 1/16 & 10 & 160 \\
Lorenz Attractor & LA & \cite{Strogatz2014} & 3 & 1 & 5/3 & N & 1/256 & 1 & 256 \\
R\"{o}ssler Attractor & RA & \cite{Strogatz2014} & 3 & 1 & 5/3 & Y & 1/128 & 12 & 1536 \\
Jerk Eq. 16 & J16 & \cite{Sprott1997} & 3 & 1 & 4/3 & Y & 1/16 & 10 & 160 \\
DC-DC Converter & DC & \cite{RunggerZamani2018} & 2 & 2 & 1 & N & 1/10 & 5 & 50 \\
\bottomrule
\end{tabular}
\caption{List of systems tested, and summary information on the experimental setup.}
\label{tbl:systems}
\end{table*}

Table \ref{tbl:systems} summarizes the properties of these systems and the experiments performed. Along with the reference to the literature, we tabulate the number of variables $n$ and inputs $m$, specify whether the inputs are additive ($+$), the step size $h$ and the evolution time $T_e$. For quick reference we also show the number of steps involved in the evolution $T_e / h$. The systems are sorted in descending value of $\bar{O}$, i.e., the average polynomial order of the differential dynamics, where $\bar{O} = 1$ implies a linear system.

In the following we complete the information on all systems by providing the dynamics, the input ranges and any overridden tool parameters used by \CORA.

\subsubsection{Higgins-Sel'kov}
\begin{align*} 
\dot{S} &=  v_0 - S k_1 P^2 \\ 
\dot{P} &=  S k_1 P^2 - k_2 P
\end{align*}

with $v_0 = 1 \pm 0.0002$, $k_1 = 1 \pm 0.0002$ and $k_2 = 1.00001 \pm 0.0002$.

\medskip

\noindent \CORA\ parameters overriding defaults:
\begin{itemize}
\item \texttt{zonotopeOrder}: inf
\item \texttt{tensorOrder}: 2.
\end{itemize}

\subsubsection{Chemical Reactor}
\begin{align*} 
\dot{x}_A &=  -u_3 x_A x_B - 0.4 x_A x_C + 0.05 u_1 - 0.1 x_A \\ 
\dot{x}_B &=  -u_3 x_A x_B + 0.05 u_2 - 0.1 x_B \\
\dot{x}_C &=  u_3 x_A x_B -0.4 x_A x_C -0.1 x_C \\
\dot{x}_D &=  0.4 x_A x_C -0.1 x_D \\
\end{align*}

with $u_1 = 1 \pm 0.001$, $u_2 = 0.9 \pm 0.001$ and $u_3 = 30 \pm 0.2$. With respect to \cite{HarwoodBarton2016}, input range widths have been divided by $100$ since none of the three tools were able to analyze the system otherwise.

\medskip

\noindent \CORA\ parameters overriding defaults:
\begin{itemize}
\item \texttt{tensorOrder}: 2.
\end{itemize}

\subsubsection{Lotka-Volterra}
\begin{align*} 
\dot{x} &= u_1\,x (1-y) \\
\dot{y} &= u_2\, y(x-1)
\end{align*}

with $u_1 = 3 \pm 0.01$ and $u_2 = 1 \pm 0.01$.

\medskip

\noindent \CORA\ parameters overriding defaults:
\begin{itemize}
\item \texttt{zonotopeOrder}: 10
\item \texttt{tensorOrder}: 2
\item \texttt{reductionInterval}: 50.
\end{itemize}

\subsubsection{Jet Engine}
\begin{align*} 
\dot{x} &= -y-1.5 x^2-0.5 x^3 -0.5+u_1 \\
\dot{y} &= 3 x-y+u_2
\end{align*}

with $u_1 = \pm 0.005$ and $u_2 = \pm 0.005$.

\medskip

\noindent \CORA\ parameters overriding defaults:
\begin{itemize}
\item \texttt{zonotopeOrder}: 200
\item \texttt{intermediateOrder}: 200
\item \texttt{advancedLinErrorComp}: 1.
\end{itemize}

\subsubsection{PI Controller}
\begin{align*} 
\dot{v} &= -0.101 (v-20) +1.3203(x-0.1616)-0.01 v^2 \\
\dot{x} &= 0.101 (v-20) -1.3203 (x-0.1616)+0.01 v^2 + 3(20-v) + u
\end{align*}

with $u = \pm 0.1$.

\medskip

\noindent \CORA\ parameters overriding defaults:
\begin{itemize}
\item \texttt{zonotopeOrder}: 200
\item \texttt{advancedLinErrorComp}: 1.
\end{itemize}

\subsubsection{Jerk Equation 21}
\begin{align*} 
\dot{x} &= y \\
\dot{y} &= z \\
\dot{z} &= -z^3-y x^2-u x
\end{align*}

with $u = 0.25 \pm 0.01$.

\medskip

\noindent \CORA\ parameters overriding defaults:
\begin{itemize}
\item \texttt{zonotopeOrder}: 300
\item \texttt{intermediateOrder}: 200
\item \texttt{errorOrder}: 50
\item \texttt{advancedLinErrorComp}: 1.
\end{itemize}

\subsubsection{Lorenz Attractor}
\begin{align*} 
\dot{x} &= y \\
\dot{y} &= z \\
\dot{z} &= -z^3-y x^2-u x
\end{align*}

with $u = 28 \pm 0.01$.

\medskip

\noindent \CORA\ parameters overriding defaults:
\begin{itemize}
\item \texttt{zonotopeOrder}: 300.
\end{itemize}

\subsubsection{R\"{o}ssler Attractor}
\begin{align*} 
\dot{x} &= -y-z \\
\dot{y} &= x+0.1 y \\
\dot{z} &= z (x-6)+u
\end{align*}

with $u = 0.1 \pm 0.001$.

\medskip

\noindent \CORA\ parameters overriding defaults:
\begin{itemize}
\item \texttt{tensorOrder}: 2.
\end{itemize}

\subsubsection{Jerk Equation 16}
\begin{align*} 
\dot{x} &= y \\
\dot{y} &= z \\
\dot{z} &= -y+x^2+u
\end{align*}

with $u = -0.03 \pm 0.001$.

\subsubsection{DC-DC Converter}

With respect to \cite{RunggerZamani2018}, the system has been rewritten in its equivalent input-affine form in order to be analyzed using \Ariadne:
\begin{align*} 
\dot{x} &= -0.018\, x -0.066\, y+u_1(\frac{1}{600} x+ \frac{1}{15} y)+u_2 \\
\dot{y} &= 0.071\, x -0.00853\, y+u_1(-\frac{1}{14} x - \frac{20}{7} y) \\
\end{align*}

with $u_1 = \pm 0.002$ and $u_2 = \frac{1}{3} \pm \frac{1}{15}$. 

\medskip

\noindent \CORA\ parameters overriding defaults:
\begin{itemize}
\item \texttt{taylorTerms}: 20
\item \texttt{tensorOrder}: 2.
\end{itemize}

\subsection{Results}

This subsection on results starts by evaluating the quality of approximation with/without simplification of the parameters that represent a set. After assessing the quality at the default noise levels, we analyze the effect of varying the noise levels, along with the number of parameters after a simplification and the simplification period. The next subsection will compare these results with those obtained using \CORA\ and \Flowstar.

Given the large size of the benchmark suite, figures will be shown only for selected systems on some results. Instead we will rely on quantitative tabular data based on the metrics that were previously introduced.

\begin{table*}[b]\centering
\ra{1.1}
\tabcolsep=0.10cm
\begin{tabular}{@{}lrrrcrrcrrcrrcrr@{}}
\toprule
& \phantom{a} 
& \multicolumn{2}{c}{$Z$} & \phantom{a} 
& \multicolumn{2}{c}{$C$} & \phantom{a} 
& \multicolumn{2}{c}{$A$} & \phantom{a} 
& \multicolumn{2}{c}{$S$} & \phantom{a} 
& \multicolumn{2}{c}{$P$} \\
\cmidrule{3-4} \cmidrule{6-7} \cmidrule{9-10} \cmidrule{12-13} \cmidrule{15-16} 
&& $\Sigma_V$ & $t_x$ && $\Sigma_V$ & $t_x$ && $\Sigma_V$ & $t_x$ && $\Sigma_V$ & $t_x$ && $\Sigma_V$ & $t_x$\\ 
\midrule
HS && 31.60 & 851 && 32.16 & 11143 && \multicolumn{2}{c}{T.O.} && \multicolumn{2}{c}{T.O.} && \multicolumn{2}{c}{T.O.} \\
CR && 100.1 & 99 && 195.2 & 247 && \bf{323.3} & 640 && 170.5 & 931 && 247.9 & 1149 \\
LV && 5.267 & 813 && 12.08 & 7674 && 11.27 & 26754 && \multicolumn{2}{c}{T.O.} && \multicolumn{2}{c}{T.O.} \\
JE && \bf{16.13} & 166 && 14.65 & 725 && 15.19 & 2434 && 15.14 & 2423 && 14.55 & 3243 \\
PI && 2.929 & 38 && 4.299 & 66 && 5.944 & 101 && \bf{5.959} & 105 && 5.946 & 175 \\
J21 && 15.67 & 188 && 19.86 & 223 && \bf{23.41} & 292 && 22.98 & 304 && 22.70 & 425 \\
LA && 5.144 & 311 && 8.297 & 546 && 12.14 & 1103 && \bf{12.15} & 1152 && 11.71 & 1925 \\
RA && \multicolumn{2}{c}{T.O.} && \multicolumn{2}{c}{T.O.} && \multicolumn{2}{c}{T.O.} && \multicolumn{2}{c}{T.O.} && \multicolumn{2}{c}{T.O.} \\
J16 && 14.06 & 68 && 22.04 & 108 && \bf{26.86} & 165 && 26.77 & 165 && 25.20 & 301 \\
DC && 0.909 & 11 && 1.900 & 503 && \bf{1.920} & 1130 && 1.914 & 1385 && 1.919 & 2073 \\
\bottomrule
\end{tabular}
\caption{Score $\Sigma_V$ and execution times $t_x$ in seconds for each system and each approximation, where no simplification of the parameters is performed. The best score for a given system is emphasized in bold. A timeout (T.O.) is obtained if completion is not achieved within 8 hours of execution.}
\label{tbl:systems-nosimplify-approximations}
\end{table*}

In Table~\ref{tbl:systems-nosimplify-approximations} we show the results in terms of score $\Sigma_V$ and execution time $t_x$ when using a given approximation ($Z$ for zero, $C$ for constant, $A$ for affine, $S$ for sinusoidal and $P$ for piecewise-constant). In particular, we want to evaluate performance when no resetting of the parameters is performed. Results show an interesting behavior: the best approximation in terms of volume score $\Sigma_V$ does not always stem from using the highest number of parameters for the auxiliary system (i.e., two in the case of $A$, $S$ or $P$). Namely, $C$ outperforms $A$ for LV and $S$ for CR; more interestingly, JE gives the best result using $Z$. Since it can be shown that the local error for the chosen step sizes monotonically decreases from $Z$ to $P$, the motivation lies in the representation of the flow set as a result of the addition of the auxiliary functions. Higher-order auxiliary functions influence numerical quality of integration due to the progressive addition of parameters along with the more complicated flow function to integrate. When the number of integration steps involved is significant, as for RA, a result cannot be obtained within 8 hours of execution due to the exceedingly large number of parameters, hence termination is enforced. A similar timeout is present also for systems with a lower number of steps, i.e. HS and LV, which instead feature a higher number of new parameters per step and higher nonlinearity in the dynamics (refer to Table~\ref{tbl:systems} for comparisons); in these cases the timeout is due to the effort of evolving a set with a larger number of parameters. 

Table~\ref{tbl:systems-nosimplify-approximations} also shows that even if we focus on approximations using two parameters, the best result largely depends on the system under analysis. This behavior, along with the particular case of JE, suggests that we should check all available approximations and choose the best one. Our framework allows for this choice to be performed at each integration step. However, this {\em tight} approach incurs in a significant cost in terms of execution time, slightly lower than the sum of the costs in Table~\ref{tbl:systems-nosimplify-approximations}. Consequently we defined a {\em loose} approach for choosing the best approximation: a counter $k_a$ is associated with a given approximation $a$, with $k_a = 1 \, \forall a$ at the beginning of evolution; if an approximation is not the best one, the value of $k_a$ is doubled and $a$ will be checked again after $k_a$ steps; instead when $a$ iis the best one, we reset $k_a =1$. Such exponential delay in checking a less-than-optimal approximation allows to focus on the best approximation(s). 

Table~\ref{tbl:systems-nosimplify-selection} compares the best available result for each system from Table~\ref{tbl:systems-nosimplify-approximations}, where the approximation is chosen statically at the beginning of evolution and used for all steps, with dynamic choices of the best approximation using respectively the tight and the loose approach. Since a dynamic choice will, in general, yield a mix of approximations, we provide a ``a\%" column that summarizes the frequency of choosing a given approximation, i.e., A93P7 means that the affine approximation was the best one on 93\% of the steps while the piecewise-affine approximation was chosen on the remaining 7\%. We see that a tight dynamic choice yields better results than the best static choice; our evaluation showed that the best approximation changes infrequently and we can identify sections of the evolution where a given approximation is always chosen. Therefore such behavior is compatible with a loose dynamic choice of the best approximation: as shown in the third column of Table~\ref{tbl:systems-nosimplify-selection}, the score $\Sigma_V$ is very close to the one coming from a tight approximation, while the execution time $t_x$ is not particularly higher than the one coming from the best static approximation. Still, the execution time remains significantly high, preventing completion for some of the systems. In the following we will analyze the effect of performing a periodic simplification of the parameters, with simplification period $N_s = 12$, where we keep $\beta_s = 6$ times the number of parameters introduced between simplification events.

\begin{table*}\centering
\ra{1.1}
\tabcolsep=0.15cm
\begin{tabular}{@{}lrrrccrrrcrrrc@{}}
\toprule
&& \multicolumn{3}{c}{best static} & \phantom{.} 
& \multicolumn{3}{c}{tight dynamic} & \phantom{.} 
& \multicolumn{3}{c}{loose dynamic}  \\
\cmidrule{3-5} \cmidrule{7-9} \cmidrule{11-13} 
&& $\Sigma_V$ & $t_x$ & a && $\Sigma_V$ & $t_x$ & a\% && $\Sigma_V$ & $t_x$ & a\% \\ 
\midrule
HS && 32.16 & 11143 & C && \multicolumn{3}{c}{T.O.} && \multicolumn{3}{c}{T.O.} \\
CR && 323.3 & 640 & A && \bf{324.0} & 3894 & A93P7 && 323.6 & 683 & A91P9 \\
LV && 12.08 & 7674 & C && \multicolumn{3}{c}{T.O.} && \multicolumn{3}{c}{T.O.} \\
JE && 16.13 & 166 & Z && \bf{16.16} & 1887 & Z82P18 && 16.13 & 171 & Z100 \\
PI && 5.959 & 105 & S && \bf{5.962} & 580 & S44P56 && 5.960 & 151 & S24P76 \\
J21 && 23.41 & 292 & A && \bf{23.94} & 1433 & C2A86S4P8 && 23.41 & 295 & A100 \\
LA && 12.15 & 1152 & S && \bf{12.22} & 6398 & A61S24P15 && 12.20 & 1429 & A48S37P15 \\
RA &&  \multicolumn{3}{c}{T.O.} && \multicolumn{3}{c}{T.O.} && \multicolumn{3}{c}{T.O.} \\
J16 && \bf{26.86} & 165 & A && \bf{26.86} & 901 & A96P4 && \bf{26.86} & 176 & A96P4 \\
DC && \bf{1.920} & 1130 & A && \bf{1.920} & 6534 & A97P3 && \bf{1.920} & 1268 & A96P4 \\
\bottomrule
\end{tabular}
\caption{Volume score $\Sigma_V$ and execution times $t_x$ in seconds for each system and various setups, when not simplifying the number of parameters; the first one picks the best approximation statically chosen from Table~\ref{tbl:systems-nosimplify-approximations}; the second one comes from dynamically evaluating each approximation at each step and selecting the best one; the third one comes from dynamically evaluating each approximation with a frequency proportional to its quality. The best $\Sigma_V$ for a given system is emphasized in bold. A timeout (T.O.) is obtained if completion is not achieved within 8 hours of execution.}
\label{tbl:systems-nosimplify-selection}
\end{table*}

\begin{table}\centering
\ra{1.1}
\tabcolsep=0.15cm
\begin{tabular}{@{}lrrrcrrcrrcrrcrr@{}}
\toprule
& \phantom{a}
& \multicolumn{2}{c}{$Z$} & \phantom{a}
& \multicolumn{2}{c}{$C$} & \phantom{a}
& \multicolumn{2}{c}{$A$} & \phantom{a}
& \multicolumn{2}{c}{$S$} & \phantom{a}
& \multicolumn{2}{c}{$P$} \\
\cmidrule{3-4} \cmidrule{6-7} \cmidrule{9-10} \cmidrule{12-13} \cmidrule{15-16}
&&  $\Sigma_V$ & $t_x$ &&  $\Sigma_V$ & $t_x$ &&  $\Sigma_V$ & $t_x$ &&  $\Sigma_V$ & $t_x$ &&  $\Sigma_V$ & $t_x$\\
\midrule
HS && 30.80 & 84 && 46.56 & 35 && \bf{48.40} & 38 && 41.76 & 131 && 44.17 & 47  \\
CR && 101.2 & 21 && 214.3 & 14 && \bf{502.3} & 21 && 219.6 & 146 && 428.9 & 23 \\
LV && 5.265 & 219 && 10.89 & 89 && \bf{14.53} & 60 && 12.83 & 169 && 13.53 & 76 \\
JE && \bf{15.47} & 25 && 13.72 & 26 && 14.43 & 26 && 14.37 & 54 && 14.56 & 28 \\
PI && 2.701 & 6.7 && 3.859 & 5.6 && 5.486 & 5.5 && 5.479 & 10 && \bf{5.492} & 7.8 \\
J21 && 15.08 & 31 && 19.37 & 17 && 23.10 & 13 && 22.90 & 20 && \bf{23.23} & 15 \\
LA && 1.325 & 41 && 6.187 & 23 && 8.979 & 14 && 8.992 & 19 && \bf{9.045} & 18 \\
RA && 71.70 & 46 && 107.2 & 28 && 114.2 & 25 && 109.8 & 34 && \bf{120.0} & 36 \\
J16 && 12.00 & 14 && 19.49 & 6.5 && 23.78 & 5.3 && \bf{23.78} & 6.2 && 23.27 & 7.7  \\
DC && 0.907 & 2.7 && 1.888 & 5.4 && \bf{1.906} & 5.9 && 1.902 & 13 && 1.906 & 11 \\
\bottomrule
\end{tabular}
\caption{Score $\Sigma_V$ and execution times $t_x$ in seconds for each system and each approximation, where simplification of the parameters is performed. The highest score for a given system is emphasized in bold.}
\label{tbl:systems-simplify-approximations}
\end{table}

Table~\ref{tbl:systems-simplify-approximations} shows the results when using simplification. Compared with Table~\ref{tbl:systems-nosimplify-approximations}, it is apparent that the execution times are significantly reduced. This in turn allows to complete execution for all approximations on all systems. The best static approximation for a given system differs in the presence of simplification, but this situation is somehow expected due to different set volumes and number of parameters involved. Figure~\ref{fig:approximations} shows the trajectory of the CR system, specifically on the $x_A$-$x_C$ projection, comparing the results with no simplification (left figure) and with simplification (right figure). For graphical purposes, the trajectories are overlapped, drawn from the coarsest to the finest, from black to white, in order to show the different flow radiuses; the initial values are $(0,0)$ and we see how the trajectory increases its radius in the two cases.

\begin{figure*}[htp]
   \subfloat[No simplification]{\label{fig:approximations-no-simplify} 
      \includegraphics[width=.47\textwidth]{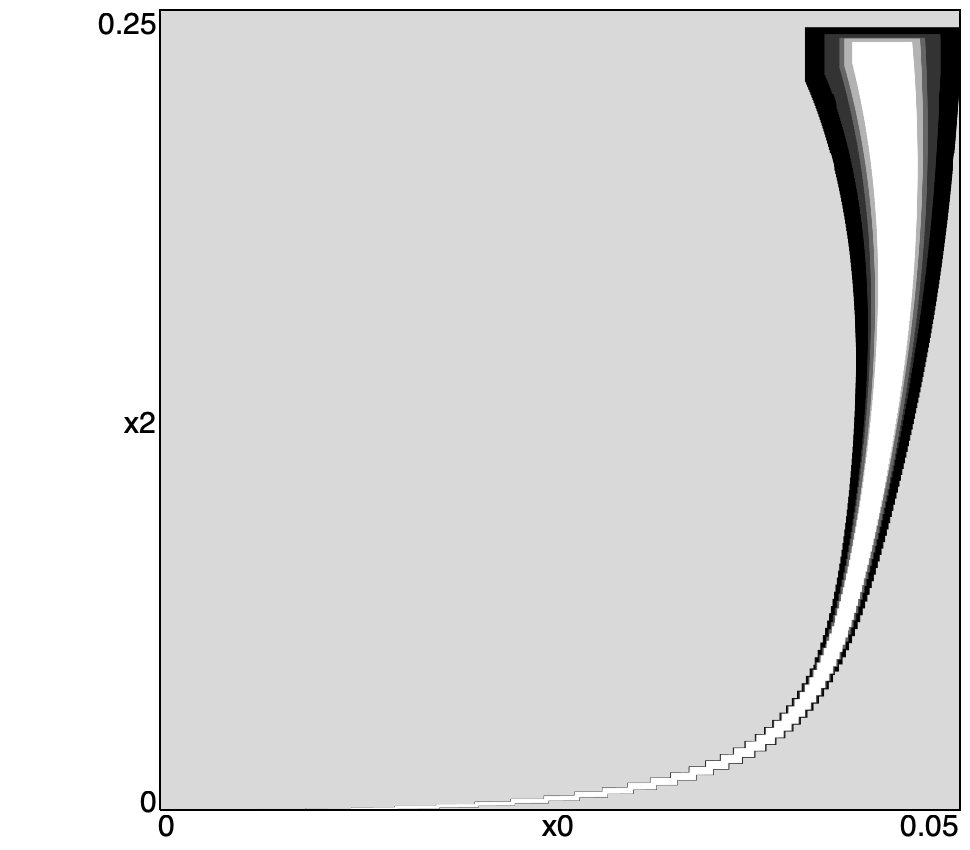}}
~
   \subfloat[Simplification]{\label{fig:approximations-simplify}
      \includegraphics[width=.47\textwidth]{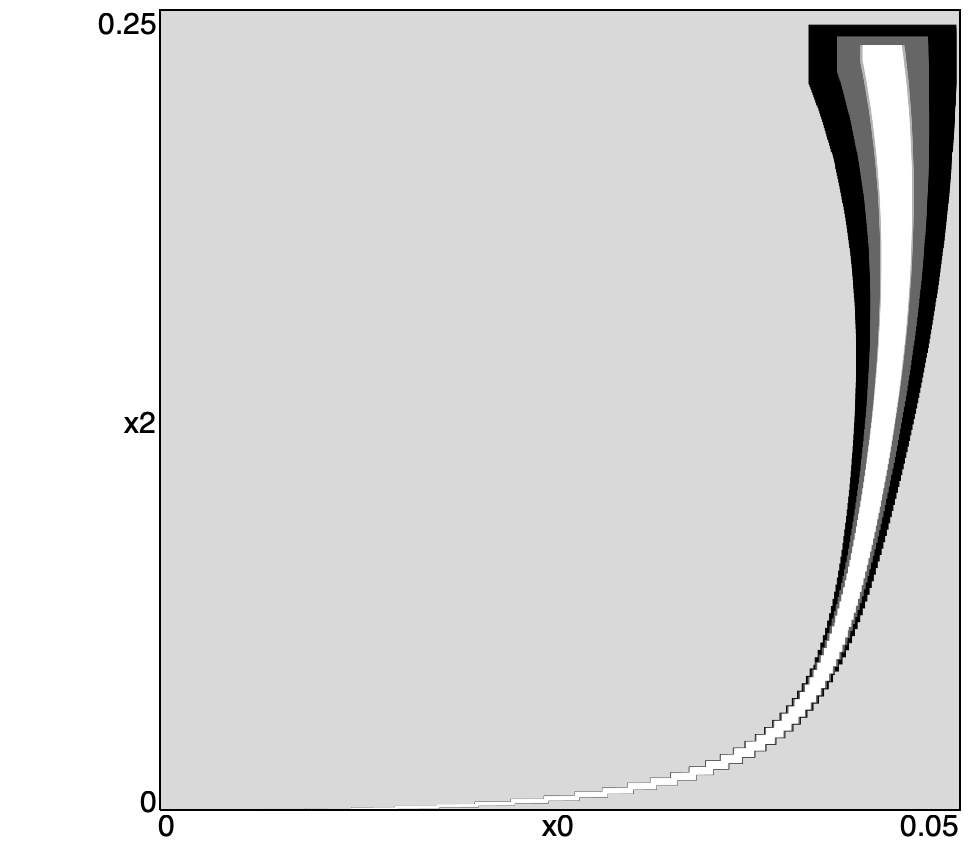}}

   \caption{Plot of the $x_A$-$x_C$ trajectory of the CR system for all approximations, drawn starting from the coarsest one (black fill) to the finest one (white fill), either with no simplification of the parameters (\ref{fig:approximations-no-simplify}) or with simplification (\ref{fig:approximations-simplify}).}
   \label{fig:approximations}
\end{figure*}

In order to evaluate the complete benchmark suite on a dynamic choice of the best approximation, Table~\ref{tbl:systems-simplify-selection} provides data equivalent to Table~\ref{tbl:systems-nosimplify-selection}. On the first column we also tabulate the best loose dynamic result from Table~\ref{tbl:systems-nosimplify-selection} itself for comparison purposes. We notice that the volume score metric, being inaccurate, can sometimes result in unexpected behaviors, such as for LV a loose dynamic score higher than the tight dynamic score, or for JE a tight dynamic score worse than the best static score. Apart from these outliers, we can draw conclusions similar to those of Table~\ref{tbl:systems-nosimplify-selection}. Comparison with the first column shows that in some cases (i.e., at least CR and J21, if we do not consider the improvement from timeout in the HS, LV and RA cases) simplification yields a better score. This is especially true for a tight dynamic choice of the approximation, but again a loose dynamic choice allows for significantly shorter execution times with very small losses of accuracy.

\begin{table*}\centering
\ra{1.1}
\tabcolsep=0.04cm
\begin{tabular}{@{}lrrrrrrrccrrrcrrrc@{}}
\toprule
& \phantom{.} 
& \multicolumn{3}{c}{loose dynamic (no simpl.)} & \phantom{.} 
& \multicolumn{3}{c}{best static} & \phantom{.} 
& \multicolumn{3}{c}{tight dynamic} & \phantom{.} 
& \multicolumn{3}{c}{loose dynamic}  \\
\cmidrule{3-5} \cmidrule{7-9} \cmidrule{11-13} \cmidrule{15-17} 
&& $\Sigma_V$ & $t_x$ & a\% && $\Sigma_V$ & $t_x$ & a && $\Sigma_V$ & $t_x$ & a\% && $\Sigma_V$ & $t_x$ & a\% \\ 
\midrule
HS && \multicolumn{3}{c}{T.O.} && 48.40 & 38 & A && \bf{49.49} & 242 & A88P12 && 48.91 & 39 & A94P6 \\
CR && 323.6 & 683 & A91P9 && 502.3 & 21 & A && \bf{504.5} & 181 & A91P9 && 502.4 & 26 & A91P9 \\
LV && \multicolumn{3}{c}{T.O.} && 14.53 & 60 & A && 14.53 & 366 & A95P5 && \bf{14.54} & 62 & A94P6 \\
JE && \bf{16.13} & 171 & Z100 && \underline{15.47} & 25 & Z && 14.39 & 155 & Z78P21 && \underline{15.47} & 28 & Z100 \\
PI && \bf{5.960} & 151 & S24P76 && 5.492 & 7.8 & P && \underline{5.493} & 30 & S15P85 && 5.492 & 8.8 & P100 \\
J21 && 23.41 & 295 & A100 && 23.23 & 15 & P && \bf{23.77} & 63 & C1A86P13 && 23.10 & 14 & A100 \\
LA && \bf{12.20} & 1429 & A48S37P15 && 9.045 & 18 & P && \underline{9.080} & 70 & A58S6P36 && 9.070 & 18 & A46S4P50 \\
RA && \multicolumn{3}{c}{T.O.} && \bf{120.0} & 36 & P && 117.7 & 143 & A96P4 && 113.8 & 27 & A100 \\
J16 && \bf{26.86} & 176 & A96P4 && \underline{23.78} & 6.2 & S && 23.77 & 29 & A96P4 && 23.77 & 6.1 & A96P4 \\
DC && \bf{1.920} & 1268 & A96P4 && \underline{1.906} & 5.9 & A && \underline{1.906} & 36 & A71P29\ && \underline{1.906} & 7.7 & A88P12 \\
\bottomrule
\end{tabular}
\caption{Score $\Sigma_V$ and execution time $t_x$ in seconds for each system and various setups, when simplifying the parameters; the first one picks the loose selection entries from Table~\ref{tbl:systems-nosimplify-selection}; the second one picks the best from Table~\ref{tbl:systems-simplify-approximations}; the third one comes from dynamically evaluating each approximation at each step and selecting the best one; the fourth one comes from dynamically evaluating each approximation with a frequency proportional to its quality. The best score for a given system is emphasized in bold, while the best score when simplifying the parameters is emphasized through underlining, if not already the absolute best score.}
\label{tbl:systems-simplify-selection}
\end{table*}

\subsubsection{Dependency on the noise level}

Since the auxiliary system and the local error depend on the range of the inputs, it is interesting to study the relation between executions time, quality of the results, and the range of inputs. If we interpret inputs as noise sources, this corresponds to study how the noise level affects performance.

\begin{table*}[b]\centering
\ra{1.1}
\tabcolsep=0.15cm
\begin{tabular}{@{}lrrrcrrcrrcrrcrr@{}}
\toprule
& \phantom{.} 
& \multicolumn{2}{c}{x 1/4} & \phantom{.}
& \multicolumn{2}{c}{x 1/2} & \phantom{.}
& \multicolumn{2}{c}{nominal} & \phantom{.} 
& \multicolumn{2}{c}{x 2} & \phantom{.} 
& \multicolumn{2}{c}{x 4} \\
\cmidrule{3-4} \cmidrule{6-7} \cmidrule{9-10} \cmidrule{12-13} \cmidrule{15-16}
&& $\Sigma_V$ & $t_x$ && $\Sigma_V$ & $t_x$ && $\Sigma_V$ & $t_x$ && $\Sigma_V$ & $t_x$ && $\Sigma_V$ & $t_x$ \\ 
\midrule
HS && 109.1 & 22 && 76.77 & 27 && 48.91 & 39 && 23.36 & 107 && 11.49 & 296 \\
CR && 1573 & 13 && 943.3 & 19 && 502.4 & 26 && 217.6 & 53 && 60.87 & 223 \\
LV && 69.31 & 12 && 32.70 & 26 && 14.54 & 62 && 5.947 & 206 && 1.165 & 5032 \\
JE && 29.21 & 13 && 21.85 & 19 && 15.47 & 28 && 9.368 & 50 && 4.953 & 15 \\
PI && 12.82 & 7.1 && 8.849 & 7.5 && 5.492 & 8.8 && 3.085 & 10 && 1.664 & 15 \\
J21 && 36.31 & 6.8 && 30.47 & 9.2 && 23.10 & 14 && 15.41 & 27 && 8.807 & 73 \\
LA && 33.48 & 9.5 && 17.64 & 11 && 9.070 & 18 && 4.574 & 35 && 2.255 & 71 \\
RA && 385.6 & 18 && 221.5 & 20 && 113.8 & 27 && 58.85 & 48 && 29.12 & 80 \\
J16 && 58.56 & 3.6 && 39.67 & 4.3 && 23.77 & 6.1 && 13.11 & 10 && 6.570 & 22 \\
DC && 4.877 & 3.9 && 3.816 & 5.4 && 1.906 & 7.7 && 0.944 & 15 && 0.464 & 23 \\
\bottomrule
\end{tabular}
\caption{Volume score $\Sigma_V$ and execution times $t_x$ in seconds for each system, varying the noise level with respect to the nominal value.}
\label{tbl:systems-noiselevel-results}
\end{table*}

\begin{figure*}[htp]
\centering
   \includegraphics[width=.6\textwidth]{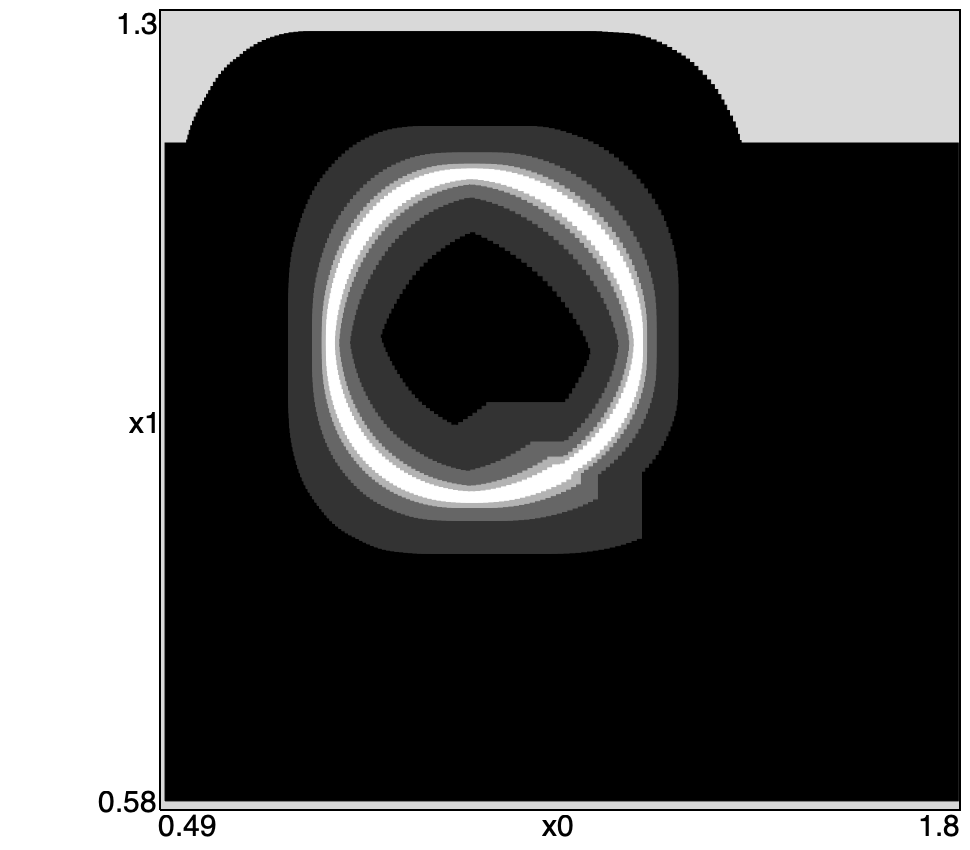}
   \caption{Plot of the trajectory of the LV system when varying the noise level, using the values in Table~\ref{tbl:systems-noiselevel-results}, from $4$ times the nominal value (black fill) to $1/4$ times the nominal value (white fill).}
   \label{fig:noise}
\end{figure*}

Table~\ref{tbl:systems-noiselevel-results} evaluates each system using a loose dynamic choice of the best approximation while simplifying the parameters. The noise level ranges from $1/4$ the nominal value to $4$ times the nominal value. Results show the expected decay in volume score when noise increases. Results also show that the execution time increases: this is due to the fact that the corresponding increase in volume of the evolved set implies a more complex polynomial representation of the set. Figure~\ref{fig:noise} shows the LV system, where we overlap the plots from the largest noise (in black) down to the smallest noise (in white). The trajectory resembles an ellipsoid, with evolution in the counterclockwise direction; since with the largest noise the reachable set increases very quickly, the black fill covers a large region with respect to the white fill.

\subsubsection{Dependency on the number of parameters after a simplification}

Now we evaluate the impact of varying the number of parameters preserved after a simplification event, hereby called $P_s$. In particular we correlate such number to the number of parameters added between simplification events, which we call $P$. We introduced $\beta_s$ as a positive number such that $P_s = \beta_s P$, where $\beta_s \geq 1$ is a reasonable condition; consequently, the number of parameters grows from $\beta_s P$ to $(\beta_s +1) P$ between simplification events. The higher $\beta_s$, the higher the average number of parameters used and therefore the more accurate the representation.

In Table~\ref{tbl:systems-beta-results} we vary between $\beta_s = 1$ to $\beta_s = 24$, with $\beta_s = +\infty$ meaning that no simplification is performed. We would expect $\Sigma_V$ to be monotonically increasing with respect to $\beta_s$, but we already know that not resetting can be detrimental to the quality. The Table indeed shows that for CR, J21 and for those systems unable to complete evolution without simplification (i.e., HS, LV and RA) the volume score has a maximum at a finite $\beta_s$. Even more interestingly, J21 seems to have multiple maxima, suggesting that small variations of the simplification policy perturb the optimal solution in a non-negligible way.

Figure~\ref{fig:beta} shows the $x$-$y$ trajectory of the LA system for the different values of $\beta_s$. Evolution starts from $(0,0)$ and ends in the bottom left corner of the figure. Compared to previous figures, the volume of the set with respect to the range of evolution in the continuous space yields a comparatively thinner trajectory, yet we can still identify the darker outline due to the smaller values of $\beta_s$.

\begin{table*}\centering
\ra{1.1}
\ra{1.1}
\tabcolsep=0.06cm
\begin{tabular}{@{}lrrrcrrcrrcrrcrrcrrcrr@{}}
\toprule
& \phantom{.} 
& \multicolumn{2}{c}{$\beta_s = 1$} & \phantom{.}
& \multicolumn{2}{c}{$\beta_s = 3$} & \phantom{.} 
& \multicolumn{2}{c}{$\beta_s = 6$} & \phantom{.} 
& \multicolumn{2}{c}{$\beta_s = 12$} & \phantom{.} 
& \multicolumn{2}{c}{$\beta_s = 18$} & \phantom{.} 
& \multicolumn{2}{c}{$\beta_s = 24$} & \phantom{.} 
& \multicolumn{2}{c}{$\beta_s = +\infty$} \\
\cmidrule{3-4} \cmidrule{6-7} \cmidrule{9-10} \cmidrule{12-13} \cmidrule{15-16} \cmidrule{18-19} \cmidrule{21-22}
&& $\Sigma_V$ & $t_x$ && $\Sigma_V$ & $t_x$ && $\Sigma_V$ & $t_x$ && $\Sigma_V$ & $t_x$ && $\Sigma_V$ & $t_x$ && $\Sigma_V$ & $t_x$ && $\Sigma_V$ & $t_x$\\ 
\midrule
HS && 17.68 & 31 && 40.20 & 34 && 48.91 & 39 && 54.30 & 54 && \bf{54.51} & 73 && 54.40 & 98 && \multicolumn{2}{c}{T.O.} \\
CR && 325.3 & 22 && 457.4 & 23 && \bf{502.4} & 26 && 467.2 & 35 && 439.1 & 45 && 415.9 & 56 && 323.6 & 683 \\
LV && 8.587 & 29 && 14.07 & 39 && \bf{14.54} & 62 && 14.31 & 112 && 13.69 & 204 && 13.33 & 291 && \multicolumn{2}{c}{T.O.}\\
JE && 14.92 & 20 && 15.27 & 23 && 15.47 & 28 && 15.64 & 39 && 15.80 & 51 && 15.89 & 63 && \bf{16.13} & 171 \\
PI && 4.945 & 5.7 && 5.102 & 6.7 && 5.492 & 8.8 && 5.574 & 12 && 5.786 & 17 && 5.823 & 22 && \bf{5.960} & 151 \\
J21 && 15.85 & 7.9 && 23.29 & 9.8 && 23.10 & 14 && \bf{23.59} & 20 && 23.16 & 29 && 23.54 & 34 && 23.41 & 295 \\
LA && 5.073 & 9.8 && 7.965 & 14 && 9.070 & 18 && 9.044 & 26 && 9.906 & 35 && 10.96 & 44 && \bf{12.20} & 1429  \\
RA && 21.42 & 16  && 58.87 & 21 && 113.8 & 27 && \bf{135.2} & 39 && 133.7 & 60 && 135.1 & 84 &&  \multicolumn{2}{c}{T.O.} \\
J16 && 15.57 & 3.3 && 21.67 & 4.7 && 23.77 & 6.1 && 25.19 & 8.3 && 25.71 & 12 && 26.46 & 16 && \bf{26.86} & 176 \\
DC && 1.888 & 4.9 && 1.905 & 5.4 && 1.906 & 7.7 && 1.906 & 13 && 1.906 & 24 && 1.909 & 44 && \bf{1.920} & 1268 \\
\bottomrule
\end{tabular}
\caption{Volume score $\Sigma_V$ and execution times $t_x$ in seconds for each system, varying the amount of parameters to keep after a simplification represented by $\beta_s$. A loose selection of the best approximation is enforced. The best $\Sigma_V$ with respect to $\beta_s$ for a given system is emphasized in bold.}
\label{tbl:systems-beta-results}
\end{table*}

\begin{figure*}[htp]
\centering
   \includegraphics[width=.6\textwidth]{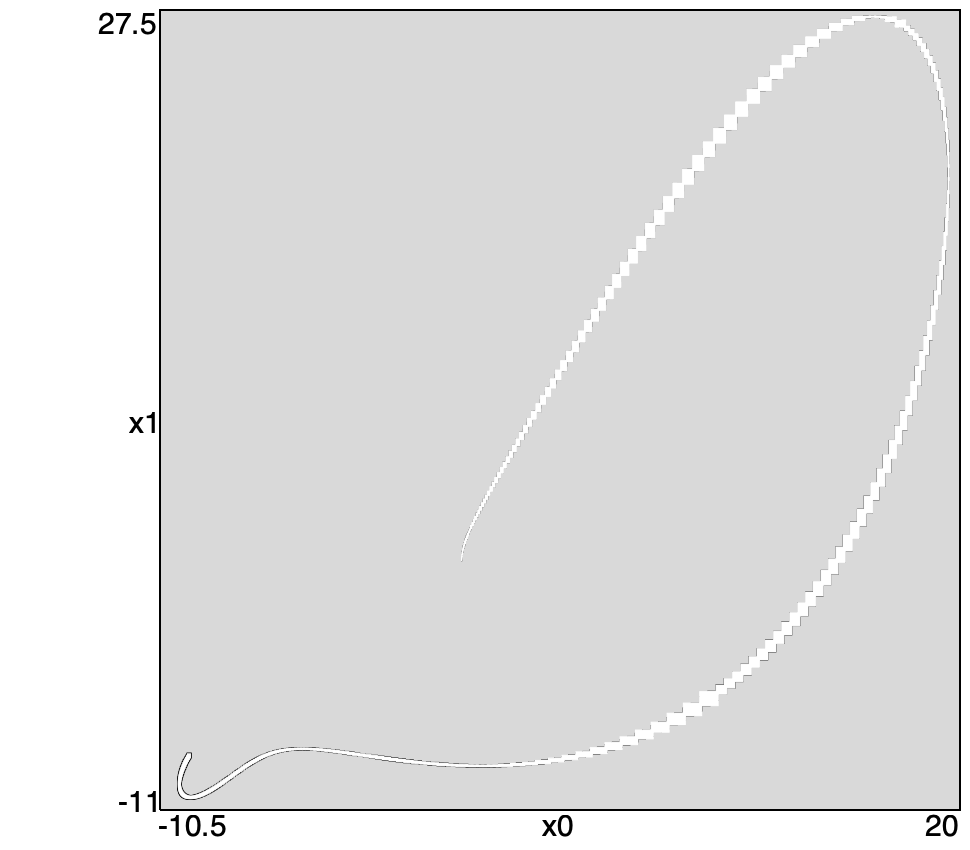}
   \caption{Plot of the $x$-$y$ trajectory of the LA system when varying the number of parameters after a simplification $\beta_s$, using the values in Table~\ref{tbl:systems-beta-results}, from a minimum equal to the number of parameters introduced between simplifications (black fill) to a maximum of infinity (white fill), meaning that no simplification is performed.}
   \label{fig:beta}
\end{figure*}

\subsubsection{Dependency on the simplification period}

The second parameter that affects the simplification policy is the simplification period $N_s$, i.e., a fixed number of steps after which a simplification event occurs. Similarly to $\beta_s$, a larger simplification period implies a larger number of parameters being used throughout evolution.

Table~\ref{tbl:systems-simplifysteps-results} shows the quality while varying $N_s$ from a value of 1, meaning simplification at each step, to $N_s = +\infty$, i.e., no simplification; for this Table we chose $N_s$ values that are relative to the total number of steps $N_{ev}$ of the specific system. Similarly to Table~\ref{tbl:systems-beta-results}, the HS, CR, LV, J21 and RA systems feature a maximum value of the score for a finite $N_s$, which leads to the same conclusion about a large number of parameters negatively affecting the approximation quality.

Figure~\ref{fig:Ns} plots the $x$-$y$ trajectories of the J16 system for all chosen values of $N_s$, overlapping from $N_s = 1$ to $N_s = +\infty$ since the quality on this system increases monotonically with $N_s$. Evolution starts from the right side of the figure and consequently the flow progressively increases along time, clearly showing the difference in volume across the different values of $N_s$.

\begin{table*}\centering
\ra{1.1}
\tabcolsep=0.045cm
\begin{tabular}{@{}lrrrcrrcrrcrrcrrcrrcrr@{}}
\toprule
& \phantom{.} 
& \multicolumn{2}{c}{$N_s = 1$} & \phantom{.}
& \multicolumn{2}{c}{$N_s = \frac{N_{ev}}{32}$} & \phantom{.}
& \multicolumn{2}{c}{$N_s = \frac{N_{ev}}{16}$} & \phantom{.}
& \multicolumn{2}{c}{$N_s = \frac{N_{ev}}{8}$} & \phantom{.} 
& \multicolumn{2}{c}{$N_s = \frac{N_{ev}}{4}$} & \phantom{.} 
& \multicolumn{2}{c}{$N_s = \frac{N_{ev}}{2}$} & \phantom{.} 
& \multicolumn{2}{c}{$N_s = +\infty$} \\
\cmidrule{3-4} \cmidrule{6-7} \cmidrule{9-10} \cmidrule{12-13} \cmidrule{15-16} \cmidrule{18-19} \cmidrule{21-22}
&& $\Sigma_V$ & $t_x$ && $\Sigma_V$ & $t_x$ && $\Sigma_V$ & $t_x$ && $\Sigma_V$ & $t_x$ && $\Sigma_V$ & $t_x$ && $\Sigma_V$ & $t_x$ && $\Sigma_V$ & $t_x$ \\ 
\midrule
HS && 15.47 & 67 && \bf{50.69} & 50 && 50.02 & 69 && 49.26 & 153 && 42.17 & 669 && 27.08 & 6163 && \multicolumn{2}{c}{T.O.} \\
CR && 199.2 & 23 && 408.3 & 21 && 506.4 & 24 && \bf{508.2} & 31 && 470.2 & 52 && 405.5 & 139 && 323.6 & 683 \\
LV && 5.563 & 94 && \bf{14.25} & 74 && 14.29 & 109 && 14.12 & 221 && 12.76 & 745 && 12.29 & 4605 && \multicolumn{2}{c}{T.O.} \\
JE && 10.29 & 25 && 15.17 & 23 && 15.42 & 30 && 15.56 & 56 && 16.00 & 128 && 16.13 & 165 && \bf{16.13} & 171 \\
PI && 4.175 & 6.4 && 5.025 & 6.7 && 5.514 & 7.8 && 5.403 & 10 && 5.241 & 18 && 5.462 & 45 && \bf{5.960} & 151 \\
J21 && 13.98 & 16 && 22.14 & 12 && 23.53 & 12 && 23.60 & 16 && \bf{23.68} & 26 && 21.73 & 72 && 23.41 & 295 \\
LA && 4.542 & 20 && 8.985 & 14 && 9.277 & 20 && 9.066 & 36 && 9.550 & 84 && 10.12 & 285 && \bf{12.20} & 1429 \\
RA && 82.00 & 26 && 141.9 & 64 && \bf{144.2} & 172 && 134.1 & 754 && 115.2 & 5192 && \multicolumn{2}{c}{T.O.} && \multicolumn{2}{c}{T.O.}  \\
J16 && 9.794 & 8.5 && 21.07 & 5.2 && 24.08 & 5.6 && 24.31 & 7.6 && 24.44 & 14 && 24.27 & 43 && \bf{26.86} & 176 \\
DC && 1.877 & 4.6 && 1.887 & 4.2 && 1.895 & 4.4 &&1.905 & 6.6 &&1.906 & 7.3 && 1.904 & 13 && \bf{1.920} & 1268 \\
\bottomrule
\end{tabular}
\caption{Volume score $\Sigma_V$ and execution times $t_x$ in seconds for each system, varying the number of steps between simplifications $N_s$ as a fraction of the total number of evolution steps $N_{ev}$. The best $\Sigma_V$ with respect to $N_s$ for a given system is emphasized in bold.}
\label{tbl:systems-simplifysteps-results}
\end{table*}

\begin{figure*}[htp]
\centering
   \includegraphics[width=.6\textwidth]{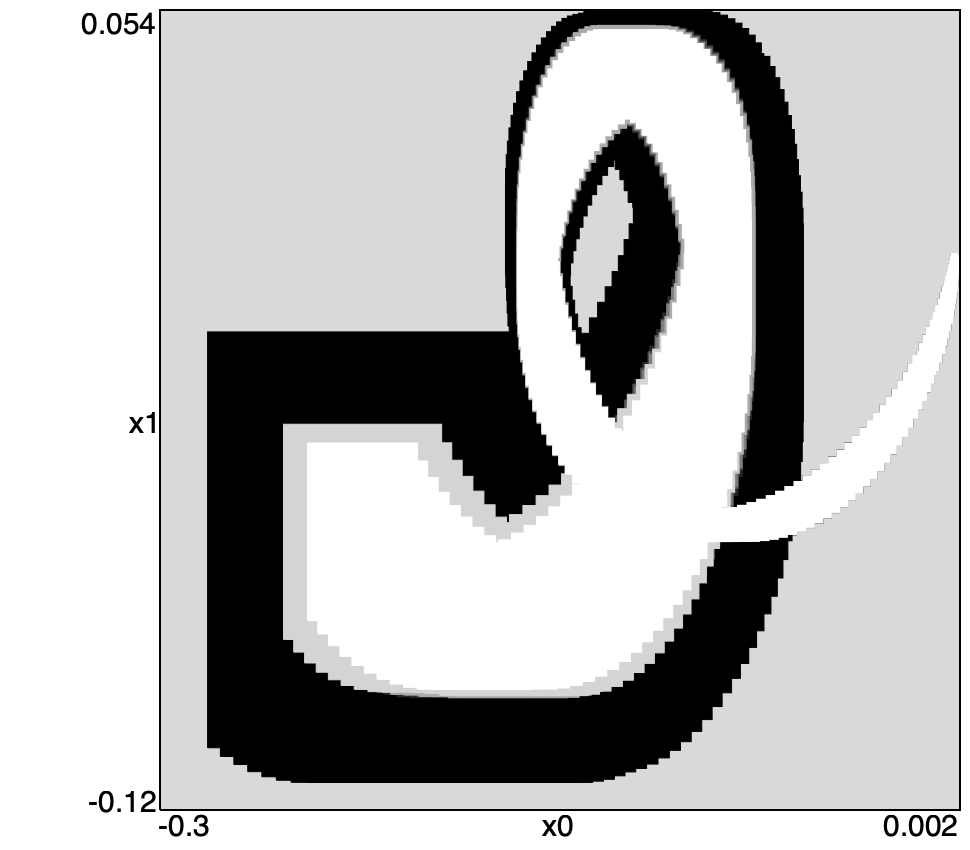}
   \caption{Plot of the $x$-$y$ trajectory of the J16 system when varying the number of steps between simplifications $N_s$, using the values in Table~\ref{tbl:systems-simplifysteps-results}, from a minimum equal to $1$ (black fill) to a maximum of infinity (white fill), meaning that no simplification is performed.}
   \label{fig:Ns}
\end{figure*}

\subsection{Comparison with other tools}

In this subsection we finally compare our results with those from \CORA\ and \Flowstar. However, since \CORA\ performs approximate rounding, its numerical results cannot be rigorous even when using interval arithmetics. For this reason, in the following Tables the actual comparison is between \Ariadne\ and \Flowstar, while \CORA\ is used as a reference.

\begin{table*}\centering
\caption{Comparison with \CORA\ and \Flowstar\ for different noise levels. For each approach and each system, the score $\Sigma_V$ is shown. Since the execution time $t_x$ is the same for \Flowstar\ regardless of the noise level, it is shown only for the nominal noise. The highest score between \Ariadne\ and \Flowstar\ for each system and each noise level is emphasized in bold. When \CORA\ produces the best result, it is underlined for reference.}
	\label{tbl:comparison}
	\ra{1.1}
	\tabcolsep=0.065cm
	\begin{tabular}{@{}cccrrrrrrrrrrr@{}}
		\toprule
		\multicolumn{2}{c}{setup} &  & &
		\multicolumn{10}{c}{system}
		\\
		\cmidrule{1-2} \cmidrule{5-14}
		noise & tool & & \phantom{.} & HS & CR & LV & JE & PI & J21 & LA & RA & J16 & DC \\
		\midrule
		\multirow{5}{*}{$\times\frac{1}{4}$} & \Ariadne && $\Sigma_V$ & \bf{109.1} & \bf{1573} & \bf{69.31} & \bf{29.21} & \bf{12.82} & \bf{36.31} & \bf{33.48} & \bf{385.6} & \bf{58.56} & 4.877 \\
		& && $t_x$ & 22 & 13 & 12 & 13 & 7.1 & 6.8 & 9.5 & 18 & 3.6 & 3.9 \\
		& \CORA && $\Sigma_V$ & 16.92 & \underline{2539} & 14.39 & 18.40 & 11.53 & 7.459 & 11.08 & 264.0 & 51.47 & \underline{7.605} \\
		& && $t_x$ & 4.0 & 1.0 & 2.5 & 3.8 & 2.5 & 3.3 & 4.0 & 2.7 & 3.7 & 0.26 \\
		& \Flowstar && $\Sigma_V$ & 71.78 & 762.1 & 2.242 & 23.18 & 11.10 & 15.75 & 17.14 & 263.5 & 52.96 & \bf{7.559} \\
		\cmidrule{1-14}
		\multirow{5}{*}{$\times\frac{1}{2}$} & \Ariadne && $\Sigma_V$ & \bf{76.77} & \bf{943.3} & \bf{32.70} & \bf{21.85} & \bf{8.849} & \bf{30.47} & \bf{17.64} & \bf{221.5} & \bf{39.67} & \bf{3.816} \\
		& && $t_x$ & 27 & 19 & 26 & 19 & 7.5 & 9.2 & 11 & 20 & 4.3 & 5.4 \\
		& \CORA && $\Sigma_V$ & 13.62 & \underline{1632} & 5.970 & 15.55 & 8.420 & 6.803 & 8.983 & 177.3 & 38.30 & \underline{3.827} \\
		& && $t_x$ & 3.9 & 1.0 & 2.6 & 3.8 & 2.3 & 6.5 & 4.1 & 4.0 & 3.5 & 0.26 \\
		& \Flowstar && $\Sigma_V$ & 56.97 & 384.5 & N/A & 19.01 & 7.994 & 14.28 & 12.33 & 174.6 & 39.11 & 3.804 \\
		\cmidrule{1-14}
		\multirow{6}{*}{$\times 1$} & \Ariadne && $\Sigma_V$ & \bf{48.91} & \bf{502.4} & \bf{14.54} & \bf{15.47} & \bf{5.492} & \bf{23.10} & \bf{9.070} & \bf{113.8} & 23.77 & \bf{1.906} \\
		& && $t_x$ & 39 & 26 & 62 & 28 & 8.8 & 14 & 18 & 27 & 6.1 & 7.7 \\
		& \CORA && $\Sigma_V$ & 8.162 & \underline{930.2} & 1.680 & 11.81 & 5.472 & 5.710 & 6.543 & 110.4 & 25.20 & \underline{1.915} \\
		& && $t_x$ & 3.9 & 1.0 & 3.5 & 3.7 & 2.5 & 6.2 & 4.1 & 4.0 & 3.3 & 0.26 \\
		& \Flowstar && $\Sigma_V$ & 37.78 & 169.9 & N/A & 13.87 & 5.107 & 11.99 & 8.113 & 107.5 & \bf{25.49} & 1.902 \\
		& && $t_x$ & 29 & 19 & 13 & 7.4 & 3.7 & 19 & 12 & 81 & 2.5 & 0.24 \\
		\cmidrule{1-14}
		\multirow{5}{*}{$\times 2$}& \Ariadne && $\Sigma_V$ & \bf{23.36} & \bf{217.6} & \bf{5.947} & \bf{9.368} & \bf{3.085} & \bf{15.41} & 4.574 & 58.85 & 13.11 & \bf{0.944} \\
		& && $t_x$ & 107 & 53 & 206 & 50 & 10 & 27 & 35 & 48 & 10 & 15 \\
		& \CORA && $\Sigma_V$ & 0.675 & \underline{433.9} & 0.807 & 7.862 & \underline{3.218} & 4.235 & 3.911 & \underline{63.67} & 14.76 & \underline{0.952} \\
		& && $t_x$ & 4.0 & 1.0 & 127 & 3.6 & 2.3 & 6.2 & 4.1 & 4.0 & 3.3 & 0.26 \\
		& \Flowstar && $\Sigma_V$ & 17.49 & 50.50 & N/A & 8.828 & 2.931 & 8.948 & \bf{4.857} & \bf{61.42} & \bf{14.76} & \bf{0.944} \\
		\cmidrule{1-14}
		\multirow{5}{*}{$\times 4$} & \Ariadne && $\Sigma_V$ & \bf{11.49} & \bf{60.87} & \bf{1.165} & \bf{4.953} & \bf{1.664} & \bf{8.807} & 2.255 & 29.12 & 6.570 & 0.464 \\
		& && $t_x$ & 296 & 223 & 5032 & 185 & 15 & 73 & 71 & 80 & 22 & 23 \\
		& \CORA && $\Sigma_V$ & N/A & \underline{146.0} & N/A & 4.517 & \underline{1.763} & 1.704 & 2.450 & \underline{33.50}& 7.825 & 0.465 \\
		& && $t_x$ & N/A & 1.0 & N/A & 3.6 & 2.2 & 6.1 & 4.0 & 4.0 & 3.4 & 0.75 \\
		& \Flowstar && $\Sigma_V$ & N/A & N/A & N/A & 4.827 & 1.577 & 5.599 & \bf{2.670} & \bf{32.23} & \bf{8.322} & \bf{0.465} \\
		\bottomrule
	\end{tabular}
\end{table*}

In Table~\ref{tbl:comparison} we evaluate the quality of our approach with respect to \Flowstar\ and \CORA\, while varying the noise level and using a fixed step size. The rationale here is that as the level increases, the impact of a more accurate input approximation increases. Systems are presented in decreasing order of nonlinearity from left to right. For mostly-linear systems \CORA\ has the best results due to its kernel relying on linearization of the dynamics; \Flowstar\ has similar benefits due to specific optimizations on low-order polynomial representations. 
On the other hand, it is apparent that \Flowstar\ and \CORA\ suffer when the nonlinearity is high, to the point of being unable to complete evolution. An N/A result in \Flowstar\ is due to failing convergence of the flow set over-approximation, while for \CORA\ this is specifically due to a diverging number of split sets required to bound the flow set. Since \Ariadne\ maintains a larger number of parameters when handling higher noise values, the computation time increases with the noise, while the computation times of \Flowstar\ and \CORA\ do not depend on the noise (Table~\ref{tbl:comparison} shows execution times only for the nominal noise). Summarizing, in this setup \Ariadne\ consistently gives better bounds for systems with medium and high nonlinearity, with comparable computation times with respect to \Flowstar\ for low noise levels, while also avoiding failure for high noise levels. 

Figure~\ref{fig:toolcomparison-noequalization} specifically compares the $x$-$y$ trajectories of the J21 system at nominal noise. Since the three tools use different plotting approaches, it was not possible to use the same canvas and we settled on enforcing the same plot range at least. It is still possible to notice from the thickness of the trajectories that \Ariadne\ has a smaller final set (on the right side of the figure) compared with \CORA\ and \Flowstar.

\begin{figure*}[htp]
   \subfloat[\Ariadne]{\label{fig:ariadne-j21-notequalized} 
      \includegraphics[width=.3\textwidth]{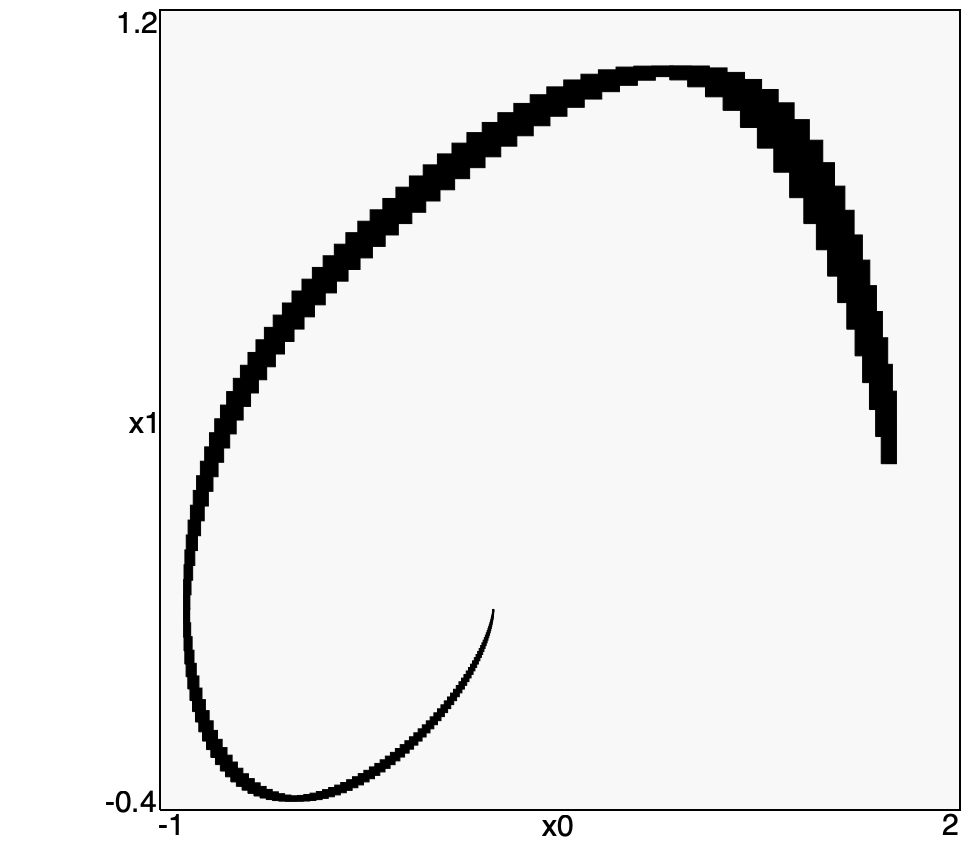}}
~
   \subfloat[\CORA]{\label{fig:cora-j21-notequalized}
      \includegraphics[width=.33\textwidth]{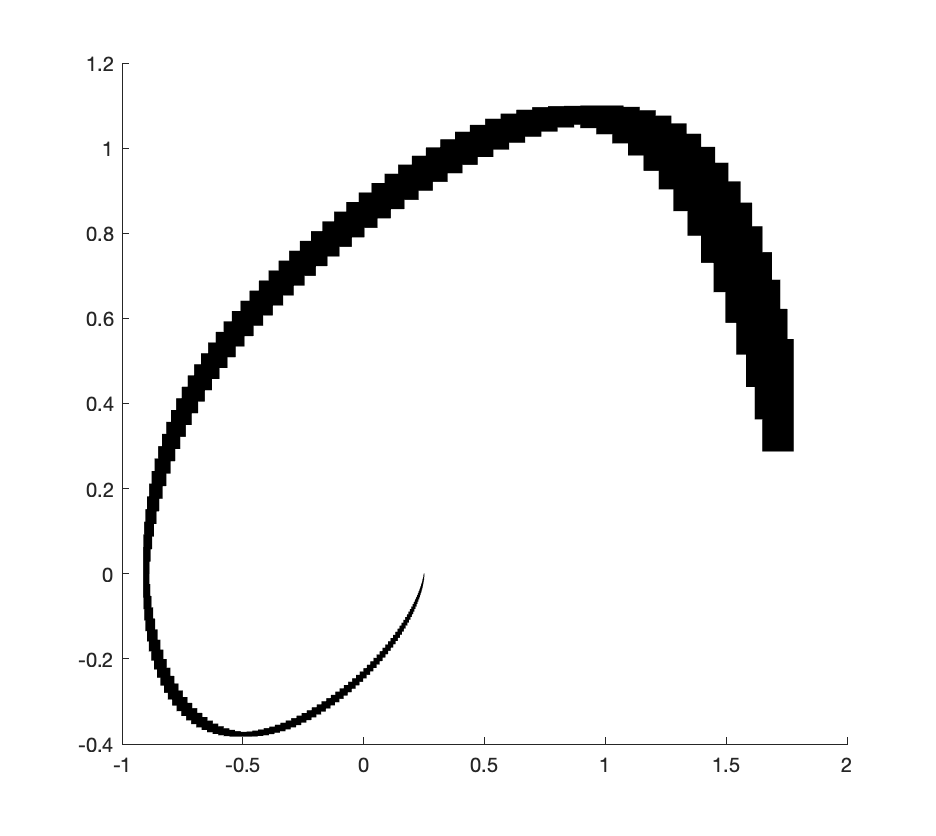}}
~
   \subfloat[\Flowstar]{\label{fig:flowstar-j21-notequalized}
      \includegraphics[width=.33\textwidth]{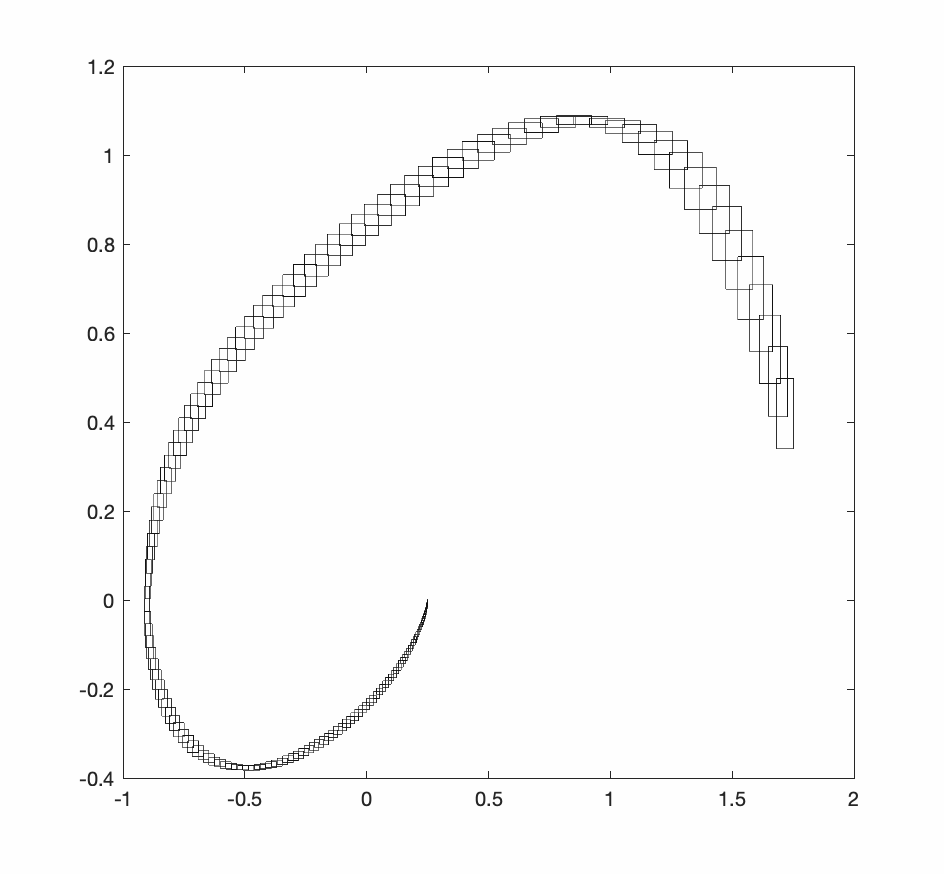}}
   \caption{Plot of the $x$-$y$ trajectory of the J21 system at nominal noise, for \Ariadne ~(\ref{fig:ariadne-j21-notequalized}), \CORA~(\ref{fig:cora-j21-notequalized}), and \Flowstar~(\ref{fig:flowstar-j21-notequalized}).}
   \label{fig:toolcomparison-noequalization}
\end{figure*}

\begin{table*}\centering
\caption{Comparison with \CORA\ and \Flowstar\ for different noise levels, while equalizing the execution time $t_x$ with respect to \Ariadne. For each approach and each system, the score $\Sigma_V$ is shown. The highest score between \Ariadne\ and \Flowstar\ for each system and each noise level is emphasized in bold. When \CORA\ produces the best result, it is underlined for reference.}
	\label{tbl:comparison-equalized}
	\ra{1.1}
	\tabcolsep=0.065cm
	\begin{tabular}{@{}cccrrrrrrrrrrr@{}}
		\toprule
		\multicolumn{2}{c}{setup} &  & &
		\multicolumn{10}{c}{system}
		\\
		\cmidrule{1-2} \cmidrule{5-14}
		noise & tool & & \phantom{.} & HS & CR & LV & JE & PI & J21 & LA & RA & J16 & DC \\
		\midrule
		\multirow{4}{*}{$\times\frac{1}{4}$} & \Ariadne && $\Sigma_V$ & \bf{109.1} & \bf{1573} & \bf{69.31} & \bf{29.21} & \bf{12.82} & \bf{36.31} & \bf{33.48} & \bf{385.6} & \bf{58.56} & 4.877 \\
		& && $t_x$ & 22 & 13 & 12 & 13 & 7.1 & 6.8 & 9.5 & 18 & 3.6 & 3.9 \\
		& \CORA && $\Sigma_V$ & 49.42 & \underline{4656} & N/A & 27.20 & \underline{13.00} & 8.753 & 12.23 & \underline{464.8} & 51.47 & 7.655 \\
		& && $\rho_h$ & 4.6 & 3.9 & N/A & 3.5 & 3.3 & 1.1 & 1.1 & 4.6 & 1.0 & 16.9 \\
		& \Flowstar && $\Sigma_V$ & 64.92 & 643.3 & 2.161 & 26.55 & 11.10 & N/A & 14.26 & 133.5 & 56.25 & \bf{7.725} \\
		& && $\rho_h$ & 0.8 & 0.7 & 0.9 & 1.7 & 3.7 & N/A & 0.8 & 0.4 & 1.3 & 9.5 \\
		\cmidrule{1-14}
		\multirow{4}{*}{$\times\frac{1}{2}$} & \Ariadne && $\Sigma_V$ & \bf{76.77} & \bf{943.3} & \bf{32.70} & 21.85 & \bf{8.849} & \bf{30.47} & \bf{17.64} & \bf{221.5} & 39.67 & 3.816 \\
		& && $t_x$ & 27 & 19 & 26 & 19 & 7.5 & 9.2 & 11 & 20 & 4.3 & 5.4 \\
		& \CORA && $\Sigma_V$ & 43.53 & \underline{2684} & N/A & \underline{22.84} & \underline{9.360} & 12.96 & 11.82 & \underline{270.6} & 40.70 & 3.820 \\
		& && $\rho_h$ & 3.0 & 4.8 & N/A & 5.3 & 3.5 & 1.6 & 1.4 & 4.8 & 1.4 & 25.0 \\
		& \Flowstar && $\Sigma_V$ & 54.76 & 384.5 & N/A & \bf{22.65} & 7.994 & N/A & 11.56 & 94.44 & \bf{42.42} & \bf{3.860} \\
		& && $\rho_h$ & 0.9 & 1.0 & N/A & 7.4 & 3.7 & N/A & 0.9 & 0.4 & 1.6 & 12.1 \\
		\cmidrule{1-14}
		\multirow{6}{*}{$\times 1$} & \Ariadne && $\Sigma_V$ & \bf{48.91} & \bf{502.4} & \bf{14.54} & 15.47 & 5.492 & \bf{23.10} & 9.070 & \bf{113.8} & 23.77 & 1.906 \\
		& && $t_x$ & 39 & 26 & 62 & 28 & 8.8 & 14 & 18 & 27 & 6.1 & 7.7 \\
		& \CORA && $\Sigma_V$ & 33.15 & \underline{1364} & N/A & \underline{16.50} & \underline{6.000} & 16.73 & \underline{9.900} & \underline{153.8} & 27.62 & 1.902 \\
		& && $\rho_h$ & 3.5 & 5.7 & N/A & 8.0 & 3.8 & 2.8 & 2.3 & 6.6 & 2.2 & 35.0 \\
		& \Flowstar && $\Sigma_V$ & 40.50 & 185.5 & N/A & \bf{16.49} & \bf{5.690} & 8.974 & \bf{9.416} & 73.96 & \bf{27.90} & \bf{1.924} \\
		& && $\rho_h$ & 1.3 & 1.3 & N/A & 3.5 & 2.3 & 0.7 & 1.5 & 0.5 & 2.1 & 14.6 \\
		\cmidrule{1-14}
		\multirow{4}{*}{$\times 2$}& \Ariadne && $\Sigma_V$ & \bf{23.36} & \bf{217.6} & \bf{5.947} & 9.368 & 3.085 & \bf{15.41} & 4.574 & \bf{58.85} & 13.11 & 0.944 \\
		& && $t_x$ & 107 & 53 & 206 & 50 & 10 & 27 & 35 & 48 & 10 & 15 \\
		& \CORA && $\Sigma_V$ & 20.07 & \underline{612.4} & N/A & 10.33 & \underline{3.507} & 15.24 & \underline{6.284} & \underline{79.61} & 15.85 & 0.944 \\
		& && $\rho_h$ & 5.8 & 8.1 & N/A & 14.0 & 4.8 & 5.3 & 4.5 & 11.5 & 3.5 & 70.0 \\
		& \Flowstar && $\Sigma_V$ & 21.13 & 69.44 & N/A & \bf{10.34} & \bf{3.316} & 10.95 & \bf{6.045} & 53.06 & \bf{16.18} & \bf{0.955} \\
		& && $\rho_h$ & 3.3 & 2.4 & N/A & 5.5 & 2.5 & 1.4 & 2.6 & 0.7 & 3.0 & 23.3 \\
		\cmidrule{1-14}
		\multirow{4}{*}{$\times 4$} & \Ariadne && $\Sigma_V$ & \bf{11.49} & \bf{60.87} & \bf{1.165} & 4.953 & 1.664 & 8.807 & 2.255 & 29.12 & 6.570 & 0.464 \\
		& && $t_x$ & 296 & 223 & 5032 & 185 & 15 & 73 & 71 & 80 & 22 & 23 \\
		& \CORA && $\Sigma_V$ & 1.086 & \underline{214.1} & N/A & 5.537 & \underline{1.909} & \underline{9.864} & 3.201 & \underline{37.97}& 6.772 & 0.465 \\
		& && $\rho_h$ & 1.1 & 17.0 & N/A & 52.0 & 7.3 & 14.6 & 9.0 & 18.1 & 7.5 & 55.0 \\
		& \Flowstar && $\Sigma_V$ & 1.725 & N/A & N/A & \bf{6.421} & \bf{1.836} & \bf{9.133} & \bf{3.286} & \bf{32.23} & \bf{8.452} & \bf{0.471} \\
		& && $\rho_h$ & 7.3 & N/A & N/A & 14.5 & 3.5 & 3.7 & 4.4 & 1.0 & 5.0 & 30.3 \\
		\bottomrule
	\end{tabular}
\end{table*}

\begin{figure*}[htp]
   \subfloat[\Ariadne]{\label{fig:ariadne-ra-equalized} 
      \includegraphics[width=.3\textwidth]{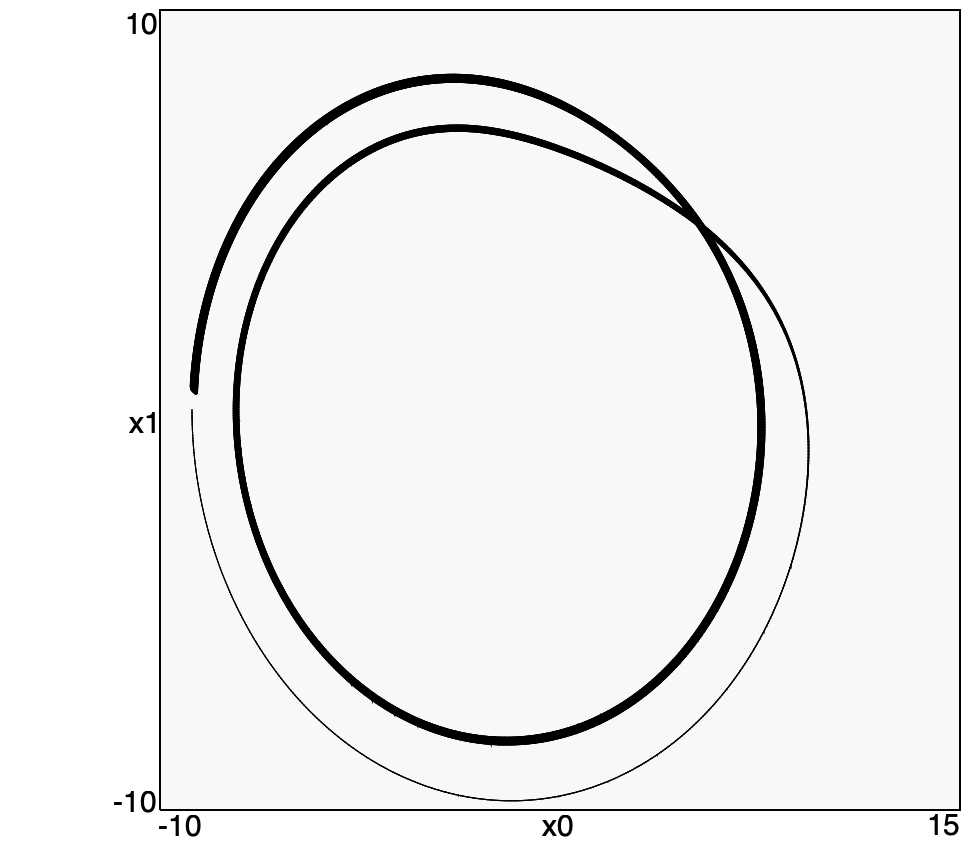}}
~
   \subfloat[\CORA]{\label{fig:cora-ra-equalized}
      \includegraphics[width=.33\textwidth]{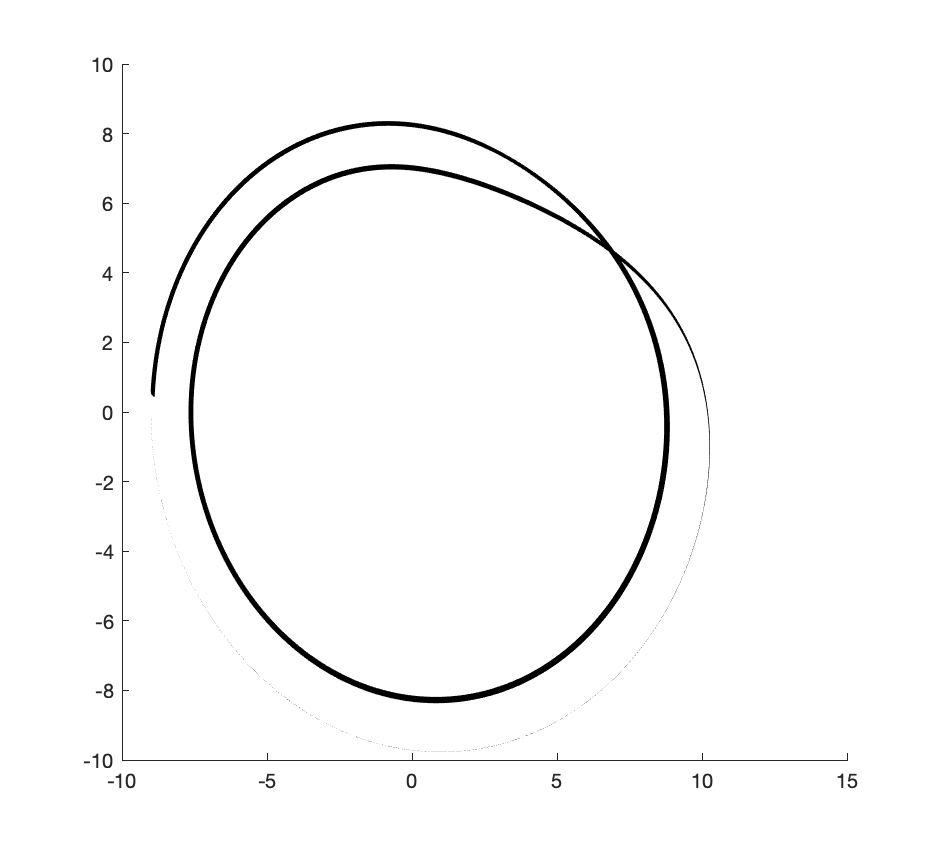}}
~
   \subfloat[\Flowstar]{\label{fig:flowstar-ra-equalized}
      \includegraphics[width=.33\textwidth]{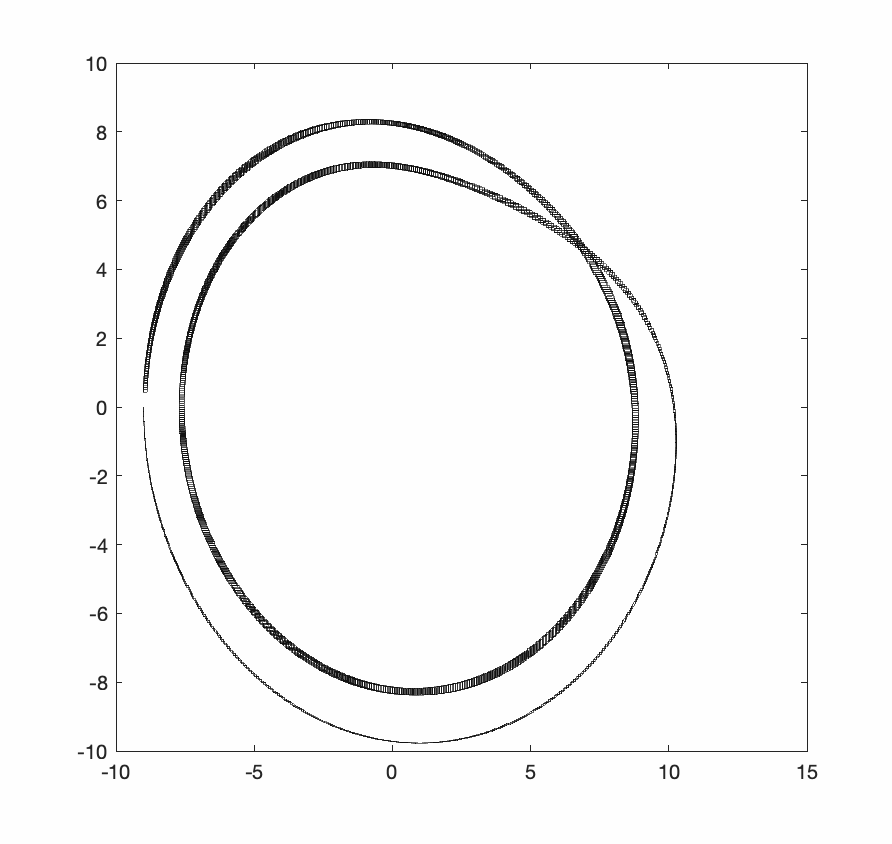}}
   \caption{Plot of the $x$-$y$ trajectory of the RA system at $4$ times the nominal noise, while equalizing the execution time, for \Ariadne ~(\ref{fig:ariadne-ra-equalized}), \CORA~(\ref{fig:cora-ra-equalized}), and \Flowstar~(\ref{fig:flowstar-ra-equalized}).}
   \label{fig:toolcomparison-equalization}
\end{figure*}

Table~\ref{tbl:comparison-equalized} compares the three tools by equalizing the execution time. This is achieved by using a different step size for \Flowstar\ and \CORA\ in order to obtain roughly the same execution time as \Ariadne. We express the ratio between the step size and the nominal step size with $\rho_h$, where $\rho_h > 0$. This approach implicitly abstracts the choice of the step size, which should be treated as a numerical setting rather than part of the system specification. Here we see that the speed advantage of \Flowstar\ on high noise levels can be actually exploited to obtain better results: here we can use $\rho_h > 1$ and obtain the best $\Sigma_V$ for low/medium nonlinearity in the dynamics. This is not the case for highly nonlinear systems, where \Flowstar\ does not converge even with a smaller step size. On the contrary, for some systems with low noise, a $\rho_h < 1$ is required for equalization, which further reduces the score with respect to \Ariadne. 
For the LV system, \CORA\ had significant issues due to splitting if the step size is reduced. Consequently, it was not actually possible to equalize the execution time.
It should be underlined that a significantly high $\rho_h$ is not without any impact: since the number of steps increases, so does the number of reachable sets. This in turn may have a non-negligible cost for operations such as set drawing, (bounded) model checking or convergence for infinite time reachability. 

Finally, Figure~\ref{fig:toolcomparison-equalization} specifically compares the $x$-$y$ trajectories of the RA system at four times the nominal noise. It can be seen from the thickness of the trajectories that \CORA\ gives a tighter approximation, followed by \Flowstar\ and finally \Ariadne.

\section{Conclusions}
\label{sec:conclusions}

In this paper, we have given a numerical method for computing rigorous over-approximations of the reachable sets of differential inclusions. The method introduces high-order error bounds for single-step approximations. By providing improved control of local errors, the method allows for accurate computation of reachable sets over longer time intervals.

We have also presented several theorems for obtaining local errors of different orders. It is easy to see that higher order errors (improved accuracy) require approximations that have a larger number of parameters (reduced efficiency). The growth of the number of parameters is an issue, in general. Sophisticated methods for handling these are at least as important as the single-step method. Nonetheless, in our evaluation of the methodology, we found that \Ariadne\ yields tighter set bounds, as the nonlinearity increases, compared with the state-of-the-art tools \Flowstar\ and \CORA. Although no analysis of the order of the method is given in~\cite{Chen2015}, we believe that \Flowstar\ has a local error $O(h^2)$, so the global error is intrinsically first-order. Hence a higher quality is to be expected from \Ariadne, since the proposed methodology is able to achieve third-order local errors. On the other hand, our approach introduces extra parameters at each step in the representation of the evolved set, causing a growth in complexity, whereas \Flowstar\ and \CORA\ have a fixed complexity of the set representations. As a result, the computational cost increases with both the noise level and the total number of steps taken. A comparison with the state-of-the-art using a common time budget indeed suggests that \Ariadne\ currently provides better bounds for highly nonlinear systems. Consequently, improving \Ariadne's methods for simplifying the description of sets represents a strategic area of ongoing research in order to fully exploit the advantage of the proposed approach.

Currently, we are working towards component-wise derivations of the local error, in order to better address systems whose variables have scaling of different orders of magnitude. Some of the other planned extensions on differential inclusions are outlined in our paper \cite{ictss2017}. These include constraint set representation of uncertainties via affine and more general convex constraints. Further, we plan an extension to nonlinearity in the inputs, to maximize the expressiveness in terms of system dynamics.

\section{ACKNOWLEDGEMENTS}
This work was partially supported by MIUR, Project ``Italian Outstanding Departments, 2018-2022" and by INDAM, GNCS 2019, ``Formal Methods for Mixed Verification Techniques".

The authors would like to thank Xin Chen and Matthias Althoff for the support on setting up their respective softwares and tuning the systems for comparison.

\bibliographystyle{spmpsci}
\bibliography{univr-biblio}

\end{document}